\documentclass[envcountsame,envcountresetchap]{svmult}

\usepackage{amscd}
\usepackage{makeidx}
\usepackage{graphicx}
\usepackage{multicol}
\usepackage[bottom]{footmisc}

\usepackage{amssymb}
\usepackage{amsmath}
\RequirePackage[T1]{fontenc}

\usepackage{longtable}
\usepackage{url}

\smartqed

\input xy
\xyoption{all}

\newcommand{\CC}{\ensuremath{\mathbb{C}}}

\newcommand{\PP}{\ensuremath{\mathbb{P}}}

\newcommand{\QQ}{\ensuremath{\mathbb{Q}}}
\newcommand{\RR}{\ensuremath{\mathbb{R}}}

\newcommand{\ZZ}{\ensuremath{\mathbb{Z}}}

\newcommand{\Spec}{\mathrm{Spec}\,}

\newcommand{\cO}{\mathcal{O}}
\newcommand{\Fuk}{\text{\sf Fuk}}

\DeclareMathOperator{\Proj}{\mathbf{Proj}}

\spnewtheorem{teo}{Theorem}[section]{\bf}{\it}
\spnewtheorem{conj}[teo]{Conjecture}{\bf}{\it}
\spnewtheorem{prop}[teo]{Proposition}{\bf}{\it}
\spnewtheorem{df}[teo]{Definition}{\bf}{\it}
\spnewtheorem{lem}[teo]{Lemma}{\bf}{\it}
\spnewtheorem{cor}[teo]{Corollary}{\bf}{\it}
\spnewtheorem{assertionp}[teo]{Assertion}{\bf}{\it}

\spnewtheorem{conclusionsp}[teo]{Conclusion}{\bf}{\rm}
\spnewtheorem{assumptionp}[teo]{Assumption}{\bf}{\rm}
\spnewtheorem{exs}[teo]{Examples}{\bf}{\rm}
\spnewtheorem{ex}[teo]{Example}{\bf}{\rm}
\spnewtheorem{ques}[teo]{Question}{\bf}{\rm}
\spnewtheorem{oss}[teo]{Remark}{\bf}{\rm}
\spnewtheorem{rmks}[teo]{Remarks}{\bf}{\rm}
\spnewtheorem{nota}[teo]{Notations}{\bf}{\rm}
\spnewtheorem{disc}[teo]{Discussion}{\bf}{\rm}
\spnewtheorem{chal}[teo]{Challenge}{\bf}{\rm}

\newlength{\nanowidth} 

\newcommand{\SarSP}{$\mathrm{II_p}$}
\newcommand{\SarDP}{$\mathrm{II_{dp}}$}
\newcommand{\SarODP}{$\mathrm{II_o}$}
\newcommand{\SarL}{$\mathrm{II_l}$}
\newcommand{\SarC}{$\mathrm{II_c}$}
\newcommand{\SarCDV}{$\mathrm{II_{cDV}}$}

\begin{document}
\setcounter{page}{1}

\title*{Birational geometry via moduli spaces}
\titlerunning{Birational geometry via moduli spaces}

\author{Ivan Cheltsov, Ludmil Katzarkov, Victor Przyjalkowski}

\author{Ivan Cheltsov, Ludmil Katzarkov, and Victor Przyjalkowski}
\authorrunning{I.\,Cheltsov, L.\,Katzarkov, V.\,Przyjalkowski}

\institute{
Ivan Cheltsov:
University of Edinburgh.\\
Ludmil Katzarkov:
University of Vienna.\\
Victor Przyjalkowski:
Steklov Mathematical Institute.\\
\texttt{I.Cheltsov@ed.ac.uk},
\texttt{lkatzark@math.uci.edu},
\texttt{victorprz@mi.ras.ru}} 

\maketitle


\begin{abstract}
In this paper we connect degenerations of Fano threefolds by projections.
Using  Mirror Symmetry we transfer these connections to the
side of Landau--Ginzburg models. Based on  that we suggest a
generalization of Kawamata's categorical approach to birational geometry
enhancing it via geometry of moduli spaces of Landau--Ginzburg models.
We suggest a conjectural application to Hasset--Kuznetsov--Tschinkel program
based on  new nonrationality ``invariants'' we consider~--- gaps and phantom categories. We make several conjectures for these invariants in the case of surfaces of general type and quadric bundles.
\end{abstract}




\

\section{Introduction}

\subsection{Moduli approach to Birational Geometry}

In recent years many significant developments have taken place in
Minimal Model Program (MMP) (see, for example, \cite{BCHM06})
based on the major advances in the study of singularities of
pairs. Similarly a categorical  approach to MMP was taken by
Kawamata. This approach was based on the correspondence Mori
fibrations (MF) and semiorthogonal decompositions (SOD). There was
no use of the discrepancies and of the effective cone in this
approach.

In the meantime   a new epoch, epoch of wall-crossing has emerged. The current situation with wall-crossing phenomenon, after  papers of Seiberg--Witten (\cite{SW}), Gaiotto--Moore--Neitzke (\cite{GMN}), Cecotti--Vafa (\cite{CV}) and
seminal works by Donaldson--Thomas (\cite{DT}), Joyce--Song (\cite{JS}),
Maulik--Nekrasov--Okounkov--Pandharipande (\cite{MNOP1}, \cite{MNOP2}), Douglas (\cite{D}), Bridgeland (\cite{B}), and
Kontsevich--Soibelman (\cite{KS1}, \cite{KS2}),  is very similar to the situation with Higgs Bundles after the works of Higgs and Hitchin --- it is clear  that general ``Hodge type'' of theory exists and needs to be developed. This lead to strong mathematical applications --- uniformization, Langlands program to mention a few. In the wall-crossing it is also clear that ``Hodge type'' of theory needs to be developed in order to reap some
mathematical benefits --- solve long standing problems in algebraic geometry.

Several steps were made to connect Homological Mirror Symmetry
(HMS) to  Birational Geometry 
(see \cite{AAK}). The main idea in this paper is that proving HMS is
like studying birational transformations (noncommutative included)
on A and B side of HMS. Later  new ideas were  introduced in
\cite{HKK}, \cite{KKP}, \cite{KP}, \cite{DKK1},  \cite{DKK2},
\cite{BFK}. These ideas can be summarized as follows:

\begin{enumerate}
  \item  The moduli spaces of stability conditions of Fukaya--Seidel categories can be included in a one-parametric family with the moduli space of Landau--Ginzburg (LG) models as a central fiber.
  \item  The moduli space of LG models determines the birational geometry --- in the toric case this was proven in  \cite{DKK1}.
\end{enumerate}

The main idea of this paper is that we consider all Mori fibered spaces together, all Sarkisov links together, all relations between Sarkisov links together. The relations between all Sarkisov links are  determined by  geometry of the moduli space of LG models  and the moving schemes involved.

We introduce  the following main ideas:

\begin{enumerate}
\item All Fano varieties via their degenerations are connected by simple basic links --- projections of a special kind. We show it (see Table~\ref{table}) on Picard rank 1 Fano threefolds. These relations agree with their toric Landau--Ginzburg models (see Theorem~\ref{theorem:toric table}).
  \item All Fano manifolds can be considered together. In other words there exists a  big moduli space of LG models, which includes mirrors of all Fanos.
We demonstrate this partially in the case of two- and three-dimensional Fanos.
  \item We introduce an analogue of the canonical divisor measure for minimal model. For us this is the local geometry of the singularities and a fiber at infinity of the LG model. The last one affects the geometry of moduli spaces of LG models --- stability conditions. In fact we propose local models  for these moduli spaces (stacks). With  this observation the correspondence between usual and categorical approach to birational geometry looks as in Table~\ref{tab:CPINTRO}.

\begin{table}[h]
  \begin{center}
    \begin{tabular}[t]{|c|c|}
\hline
\begin{minipage}[c]{1.2\nanowidth}
\centering
\medskip

Classical

\medskip

\end{minipage}
&
\begin{minipage}[c]{1.2\nanowidth}
\medskip

\begin{center}

Derived

\end{center}

\medskip

\end{minipage}
\\
\hline \hline
\begin{minipage}[c]{1.2\nanowidth}
\medskip

\begin{center}

Mori fibrations.

\end{center}

\medskip

\end{minipage}
&
\begin{minipage}[c]{1.2\nanowidth}
\medskip

\begin{center}

SOD  or circuits.

\end{center}

\medskip

\end{minipage}
\\ \hline
\begin{minipage}[c]{1.2\nanowidth}
\medskip

\begin{center}

$(K_X+\Delta)$ log differentials.

\end{center}

\medskip

\end{minipage}
&
\begin{minipage}[c]{1.2\nanowidth}
\medskip

\begin{center}

Moving schemes  of the fiber at $\infty$ in Landau--Ginzburg model described  by differentials.

\end{center}

\medskip

\end{minipage}
\\\hline
\begin{minipage}[c]{1.2\nanowidth}
\medskip

\begin{center}

Sarkisov links.

\end{center}

\medskip

\end{minipage}
&
\begin{minipage}[c]{1.2\nanowidth}
\medskip

\begin{center}

2-dimensional faces.

\end{center}

\medskip

\end{minipage}
\\ \hline
\begin{minipage}[c]{1.2\nanowidth}
\medskip

\begin{center}

Relations between Sarkisov links.

\end{center}

\medskip

\end{minipage}
&
\begin{minipage}[c]{1.2\nanowidth}
\medskip

\begin{center}

3-dimensional faces.

\end{center}

\medskip

\end{minipage}
\\ \hline
    \end{tabular}
    \caption{Extended Kawamata Program.}
    \label{tab:CPINTRO}
  \end{center}
\end{table}

\item Following \cite{DKK2} and the pioneering work \cite{BBS} we develop the notions of phantom category and we emphasize its connection with  introduced in this paper notion of a moving scheme. The last one determines the geometry of the moduli space of LG models and as a result the geometry of the initial manifold. In the case of surfaces of general type we conjecture:

\begin{description}
  \item[A] Existence of nontrivial categories in SOD 
with trivial Hochschild homology 
in the case of classical surfaces of general type, Campedelli, Godeaux, Burniat, Dolgachev surfaces and also in the categories of quotient of product of curves and fake $\PP^2$. These surfaces are not rational since they have a non-trivial fundamental group but also since they have conjecturally a quasi-phantom subcategory in their SOD. On the Landau--Ginzburg side these quasi-phantoms are described by the moving scheme. The deformation of the  Landau--Ginzburg models  is determined by the moving scheme so the quasi-phantoms factor in the geometry of the  Landau--Ginzburg models.
On the mirror side this translates to the fact the Exts between the quasi-phantom and the rest of the SOD determines the moduli space.
  \item[B] Things become even more interesting in the case of surfaces with trivial fundamental groups. We have conjectured that in the case of Barlow surface ( see~\cite{DKK2})  and rational blow-downs there exists a nontrivial category with a trivial $K^0$ group --- \emph{phantom category}.  The deformation of the  Landau--Ginzburg models  is determined by the moving scheme so the phantom factors in the geometry of the  Landau--Ginzburg models. Similarly on the mirror side this translates to the fact the Exts between the quasi-phantom and the rest of the SOD determines the moduli space.
  \item[C] We connect the existence of such phantom categories with nonrationality questions. In case of surfaces it is clear that phantoms lead to nonrationality.
In case of threefolds we exhibit examples (Sarkisov examples)  of nonrational threefolds where phantoms conjecturally imply nonrationality. We also introduce ``higher'' nonrationality  categorical invariants --- gaps of spectra, which conjecturally are not present in the  Sarkisov examples. These ideas are natural continuation of \cite{IKP}.
\end{description}

\item We introduce conjectural invariants associated to our moduli spaces --- gaps and local differentials. We suggest that these numbers (changed drastically via
``wall-crossing'') produce strong birational invariants. We relate
these invariants to so called Hasset--Kuznetsov--Tschinkel program (see~\cite{HASS},~\cite{KUZ}) ---
a program for studying rationality of four-dimensional cubic and its ``relatives''.

\end{enumerate}


\medskip

This paper is organized as follows. In Sections~\ref{section:classical} and~\ref{subsubsection:MSVHS} we relate degenerations of
Fano manifolds via projections.

Using  mirror symmetry in Section~\ref{subsubsection:MSVHS} we transfer these connections to the
side of Landau--Ginzburg model. Based on  that  in Section~\ref{section:discussion} we suggest a
generalization of Kawamata's categorical approach to birational geometry
enhancing it via the geometry of moduli spaces of Landau--Ginzburg models.
We give several applications  most notably a  conjectural application to Hasset--Kuznetsov--Tschinkel program.
Our approach is based on two categorical nonrationality invariants --- phantoms and gaps.
Full details will appear in a future paper.

\medskip

{\bf Notations.} Smooth del Pezzo threefolds (smooth Fano threefolds of index 2) we denote by $V_n$,
where $n$ is its degree with respect to a Picard group generator
except for quadric denoted by $Q$. Fano threefold of Picard rank 1, index 1,
and degree $n$ we denote by $X_n$. The rest Fano threefolds we
denote by $X_{k.m}$, where $k$ is a Picard rank of a variety and
$m$ is its number according to~\cite{IsPr99}.

Laurent polynomial from the $k$-th line of Table~\ref{table} we denote by $f_k$.
Toric variety whose fan polytope is a Newton polytope of $f_k$ we denote by $F_k$ or just a variety number {\bf k}.

We denote a Landau--Ginzburg model for variety $X$ by $LG(X)$.

\noindent
{\bf Acknowledgment.} We are very grateful to A.\,Bondal,
M.\,Ballard, C.\,Diemer, D.\,Favero, F.\,Haiden, A. Iliev, A.\,Kasprzyck,
G.\,Kerr, M.\,Kontsevich, A. Kuznetsov,  T.\,Pantev, Y.\,Soibelman, D.\,Stepanov
for the useful discussions. I.\,C. is grateful to Hausdorff Research Institute of Mathematics (Bonn) and was funded by
RFFI grants 11-01-00336-a and AG Laboratory GU-HSE, RF government
grant, ag. 11 11.G34.31.0023, L.\,K was funded by grants NSF DMS0600800, NSF FRG DMS-0652633, NSF FRG DMS-0854977, NSF DMS-0854977, NSF DMS-0901330, grants FWF P 24572-N25 and FWF P20778, and an ERC grant --- GEMIS, V.P. was funded by grants NSF FRG DMS-0854977, NSF DMS-0854977, NSF DMS-0901330, grants FWF P 24572-N25 and FWF P20778,
RFFI grants 11-01-00336-a and 11-01-00185-a, grants MK$-1192.2012.1$,
NSh$-5139.2012.1$, and AG Laboratory GU-HSE, RF government
grant, ag. 11 11.G34.31.0023.

\section{``Classical'' Birational Geometry}
\label{section:classical}
In this section we recall some facts from classical birational geometry of three-dimensional Fano varieties. We also give a new
read of this geometry making it more suitable to connect with Homological Mirror Symmetry.

\subsection{An importance of being Gorenstein.}
Among singular Fano varieties, ones with canonical Gorenstein
singularities are of special importance. They arise in many
different geometrical problems: degeneration of smooth Fano
varieties with a special regards to the problem of Mirror Symmetry
(see \cite{Batyrev1994}, \cite{Batyrev2004}), classification of
reflexive polytopes (see \cite{KreuzerSkarke1997},
\cite{KreuzerSkarke1998}), mid points of Sarkisov links and bad
Sarkisov links (see \cite{Co95}, \cite{CPR}), compactification of
certain moduli spaces (see \cite{Mukai}), etc. Historically, Fano
varieties with canonical Gorenstein singularities are the original
Fano varieties. Indeed, the name \emph{Fano varieties} originated
in the works of V.\,Iskovskikh (see \cite{Is77}, \cite{Is78}) that
filled the gaps in old results by G.\,Fano who studied in
\cite{Fa34} and \cite{Fa42} anticanonically embedded Fano
threefolds with canonical Gorenstein singularities without naming
them so (cf. \cite{Ch96}).

In dimension two canonical singularities are always Gorenstein, so
being Gorenstein is a vacuous condition. Surprisingly, the
classification of del Pezzo surfaces with canonical singularities
is simpler than the classification of smooth del Pezzo surfaces
(see \cite{dP87}, \cite{Demazure}). Fano threefolds with canonical
Gorenstein singularities are not yet classified, but first steps
in this directions are already have been made by  S.\,Mukai,
P.\,Jahnke, I.\,Radloff, I.\,Cheltsov, C.\,Shramov,
V.\,Przyjalkowski, Yu.\,Prokhorov, and I.\,Karzhemanov (see
\cite{Mukai}, \cite{JaRa06}, \cite{ChPrSh2005}, \cite{Pr05},
\cite{Karzhemanov2008}, and \cite{Karzhemanov2009}).

\subsection{Birational maps between Fano varieties and their classification.}
V.\,Iskovskikh used birational maps between Fano threefolds to
classify them. Indeed, he discovered smooth Fano variety of degree
$22$ and Picard group $\mathbb{Z}$ by constructing the following
commutative diagram:
\begin{equation}
\label{equation:V5-V22}\xymatrix{
U\ar@{->}[d]_{\alpha}\ar@{-->}[rr]^{\rho}&&W\ar@{->}[d]^{\beta}\\%
V_5&&X_{22}\ar@{-->}[ll]_{\psi}}
\end{equation}
where $V_5$ is a smooth section of the Grassmannian
$\mathrm{Gr}(2,5)\subset\mathbb{P}^9$ by a linear subspace of
codimension $3$ (they are all isomorphic), $X_{22}$ is a smooth
Fano threefold of index $1$ and degree $22$ mentioned above, i.e.
$\mathrm{Pic}(X_{22})=\mathbb{Z}[-K_{X_{22}}]$ and
$-K_{X_{22}}^3=22$, and $\alpha$ is a blow-up of the curve $C$,
$\rho$ is a flop of the proper transforms of the secant lines to
$C$, $\beta$ contracts a surface to a curve $L\subset X_{22}$ with
$-K_{X_{22}}\cdot L=1$, and $\psi$ is a double projection from the
curve $L$ (see \cite{Is77} and \cite{Is78}). This approach is very
powerful. Unfortunately, it does not always work (see
Example~\ref{example:Iskovskikh-Manin}). V.\,Iskovskikh gave many
other examples of birational maps between smooth Fano threefolds
(see \cite{IsPr99}). Later K.\,Takeuchi produced more similar
examples in \cite{Takeuchi}. Recently, P.\,Jahnke, I.\,Radloff,
and I.\,Karzhemanov produced many new examples of Fano threefolds
with canonical Gorenstein singularities by using elementary
birational transformation between them.

\subsection{Birational maps between Fano varieties and Sarkisov program.}
Results of V.\,Iskovskikh, Yu.\,Manin, V.\,Shokurov, and
K.\,Takeuchi were used by V.\,Sarkisov and A.\,Corti to create
what is now known as the three-dimensional Sarkisov program (see
\cite{Co95}). In particular, this program decomposes any
birational maps between terminal $\mathbb{Q}$-factorial Fano
threefolds with Picard group $\mathbb{Z}$ into a sequence of
so-called elementary links (often called Sarkisov links).
Recently, the three-dimensional Sarkisov program has been
generalized in higher dimensions by C.\,Hacon and J.\,McKernan
(see \cite{HaconMcKernan}).

Unfortunately, the Sarkisov program is not applicable to Fano
varieties with non-$\mathbb{Q}$-factorial singularities, it is not
applicable to Fano varieties with non-terminal singularities, and
it is not applicable to Fano varieties whose Picard group is not
$\mathbb{Z}$. Moreover, in dimension bigger than two the Sarkisov
program is not explicit except for the toric case. In dimension
three the description of Sarkisov links is closely related to the
classification of terminal non-$\mathbb{Q}$-factorial Fano
threefolds whose class group is $\mathbb{Z}^2$. In general this
problem is very far from being solved. But in Gorenstein case we
know a lot (see \cite{JaPeRa07}, \cite{JaPeRa09},
\cite{Kaloghiros}, \cite{CutroneMarshburn}, \cite{BlancLamy}).

\subsection{Basic links between del Pezzo surfaces with canonical singularities.}
\label{subsection:2dim basic links}
The anticanonical linear system $|-K_{\mathbb{P}^2}|$ gives an
embedding $\mathbb{P}^2\to\mathbb{P}^9$. Its image is a surface of
degree $9$, which we denote by $S_9$. Let $\pi\colon S_9\dasharrow
S_8$ be a birational map induced by the linear projection
$\mathbb{P}^9\dasharrow\mathbb{P}^8$ from a point in $S_9$ (the
center of the projection), where $S_8$ is surface of degree $8$ in
$\mathbb{P}^8$ obtained as the image of $S_9$ under this
projection. For simplicity, we say that $\pi_9$ is a projection of
the surface $S_9$ from a point. We get a commutative diagram
$$
\xymatrix{&\widetilde{S}_9\ar@{->}[dl]_{\alpha_9}\ar@{->}[dr]^{\beta_9}&\\%
S_9\ar@{-->}[rr]_{\pi_{9}}&&S_8,}
$$
where $\alpha_9$ is a blow-up of a smooth point of the surface
$S_9$ and $\beta_9$~is~a birational morphism that is induced by
$|-K_{\widetilde{S}_9}|$. Note that $S_8$ is a del Pezzo surface and
$(-K_{S_{8}})^2=8$.

Iterating this process and taking smooth points of the obtained
surfaces $S_i$ as centers of projections, we get the following
sequence of projections
\begin{equation}
\label{equation:projections}
\xymatrix{ \mathbb{P}^2=S_9\ar@{-->}[r]^-{\pi_9} & S_8 \ar@{-->}[r]^-{\pi_8} & S_7\ar@{-->}[r]^-{\pi_7} & S_6\ar@{-->}[r]^-{\pi_6}
& S_5\ar@{-->}[r]^-{\pi_5} & S_4\ar@{-->}[r]^-{\pi_4} & S_3,}%
\end{equation}
where every $S_i$ is a del Pezzo surface with canonical
singularities, e.g. $S_3$ is a cubic surface in $\mathbb{P}^3$
with isolated singularities that is not a cone. Note that we have
to stop our iteration at $i=3$, since the projection of $S_3$ from
its smooth point gives a rational map of degree $2$.

For every constructed projection $\pi_{i}\colon S_i\dasharrow
S_{i-1}$, we get a commutative diagram
\begin{equation}
\label{equation:basic-link}
\xymatrix{&\widetilde{S}_{i}\ar@{->}[dl]_{\alpha_{i}}\ar@{->}[dr]^{\beta_{i}}&\\%
S_i\ar@{-->}[rr]_{\pi_{i}}&&S_{i-1},}
\end{equation}
where $\alpha_i$ is a blow-up of a smooth point of the surface
$S_i$ and $\beta_i$~is~a birational morphism that is induced by
$|-K_{\widetilde{S}_i}|$. We say that the
diagram~(\ref{equation:basic-link}) is a \emph{basic link} between
del Pezzo surfaces.

Instead of $\mathbb{P}^2$, we can use an irreducible quadric as a
root of our sequence of projections. In this way, we obtain all
del Pezzo surfaces with canonical singularities except for
$\PP^2$, quadric cone and
quartic
hypersurfaces in $\mathbb{P}(1,1,1,2)$ and sextic hypersurfaces in
$\mathbb{P}(1,1,2,3)$. Note that $S_3$ is not an intersection of
quadrics (trigonal case), anticanonical linear system of every
quartic hypersurface in $\mathbb{P}(1,1,1,2)$ with canonical
singularities is a morphism that is not an embedding
(hyperelliptic case), and anticanonical linear system of every
sextic hypersurface in $\mathbb{P}(1,1,2,3)$ has a unique base
point.

Let us fix an action of a torus $(\mathbb{C}^{*})^2$ on
$\mathbb{P}^2$. So, if instead of taking smooth points as
projection centers, we take toric smooth points (fixing the torus
action), then the constructed sequence of
projections~(\ref{equation:projections}) and the commutative
diagram~(\ref{equation:basic-link}) are going to be toric as well.
In this case we say that the diagram~(\ref{equation:basic-link}) is
a \emph{toric basic link} between toric del Pezzo surfaces. Recall
that there are exactly $16$ toric del Pezzo surfaces with
canonical singularities. In fact, we can explicitly describe all
possible toric projections of toric del Pezzo surfaces from their
smooth toric points (this is purely combinatorial problem), which
also gives the complete description of all \emph{toric basic link}
between toric del Pezzo surfaces.  The easiest way of doing this
is to use reflexive lattice polytopes that correspond del Pezzo
surfaces with canonical singularities\footnote{Recall that toric
$n$-dimensional Fano varieties with canonical Gorenstein
singularities up to isomorphism are in one-two-one correspondence
with reflexive $n$-dimensional lattice polytopes in $\mathbb{R}^n$
up to $\mathrm{SL}_{n}(\mathbb{Z})$ action.}. The answer is given
by Figure~\ref{figure:del Pezzo tree}.
\begin{figure}[h]
\centering
\xymatrix{
 & & & \includegraphics[width=1cm]{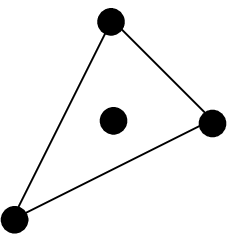} \ar[d] & & & \\
 & & \includegraphics[width=1cm]{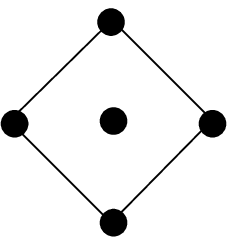} \ar[d] 
 & \includegraphics[width=1cm]{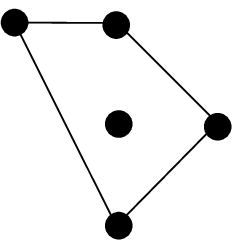} \ar[dr] \ar[dl] & \includegraphics[width=1cm]{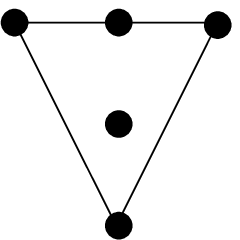} \ar[d] 
 & & \\
 & & \includegraphics[width=1cm]{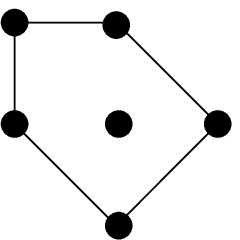} \ar[dl] \ar[drr] & & \includegraphics[width=1cm]{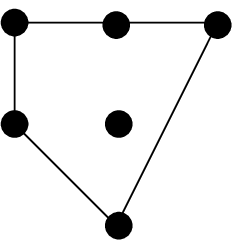} \ar[dll] \ar[dr] \ar[d] & & \\
 & \includegraphics[width=1cm]{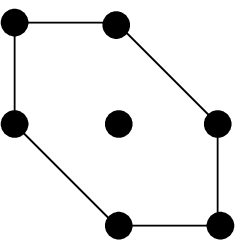} \ar[dr] & \includegraphics[width=1cm]{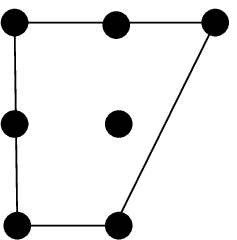} \ar[d] \ar[drr] & & \includegraphics[width=1cm]{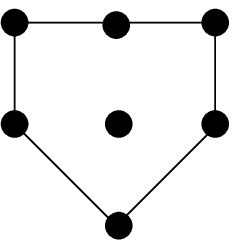} \ar[d] \ar[dll] &
\includegraphics[width=1.5cm]{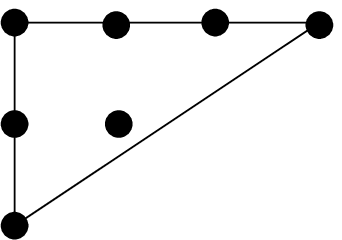} \ar[dl]\\
 & & \includegraphics[width=1cm]{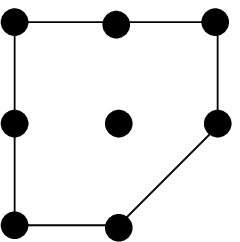} \ar[d] \ar[dr] & & \includegraphics[width=1.5cm]{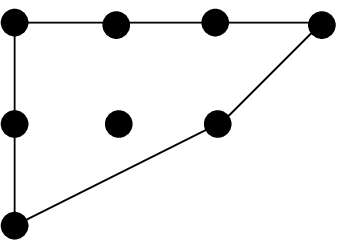} \ar[d] \ar[dl] & & \\
 & &  \includegraphics[width=1cm]{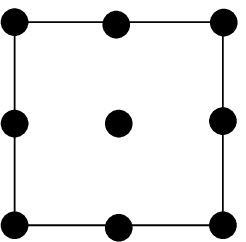} & \includegraphics[width=1.5cm]{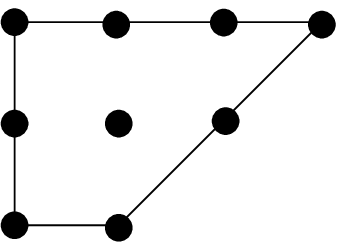} \ar[d] & \includegraphics[width=1.75cm]{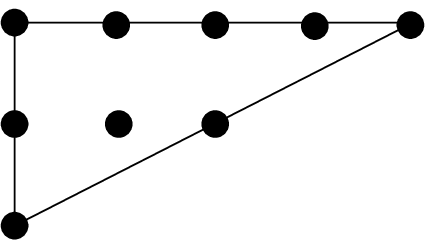} & & \\
 & & & \includegraphics[width=1.5cm]{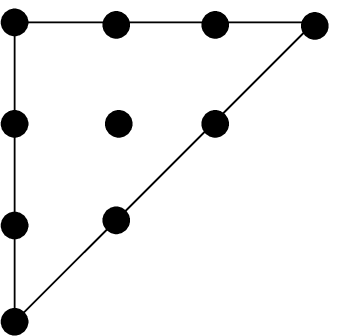} & & & \\
}
\caption{Del Pezzo tree.}\label{figure:del Pezzo tree}
\end{figure}

\subsection{Basic links between Gorenstein Fano threefolds with canonical singularities.}
Similar to the two-dimensional case, it is tempting to fix few
very explicit \emph{basic links} between Fano threefolds with
canonical Gorenstein singularities  (\emph{explicit} here means
that these basic links should have a geometric description like
projections from points or curves of small degrees, etc.) and
describe all such threefolds using these links. However, this is
impossible in general due to the following

\begin{ex}[Iskovskikh--Manin]
\label{example:Iskovskikh-Manin} The Fano threefold with canonical
singularities that is birational to a smooth quartic threefold is
a smooth quartic threefolds itself (see \cite{IsMa71} and
\cite{Cheltsov03}).
\end{ex}

However, if we are only interested in classification \emph{up to a
deformation}, then we can try to fix few very explicit \emph{basic
links} between Fano threefolds with canonical Gorenstein
singularities and describe all \emph{deformation types} of such
threefolds using these links. Moreover, it seems reasonable to
expect that this approach allows us to obtain all smooth Fano
threefolds in a unified way.

We can define three-dimensional basic links similar to the
two-dimensional case. Namely, let $X$ be a Fano threefold with canonical Gorenstein singularities.  Put $g=K_X^3/2+1$. Then
$g$ is a positive integer and
$h^{0}(\mathcal{O}_{X}(-K_{X}))=g+1$. Let $\varphi_{|-K_X|}\colon
X\to\mathbb{P}^{g+1}$ be a map given by $|-K_{X}|$. Then
\begin{enumerate}
\item either $\mathrm{Bs}|-K_X|\neq \emptyset$, and all such $X$ are found in \cite{JaRa06},%
\item or $\varphi_{|-K_X|}$ is not a morphism, the threefold $X$ is called hyperelliptic, and all such $X$ are found in \cite{ChPrSh2005},%
\item or $\varphi_{|-K_X|}$ is a morphism and  $\varphi_{|-K_X|}(X)$ is not an intersection of quadrics, the threefold $X$ is called trigonal, and all such $X$ are found in \cite{ChPrSh2005},%
\item or $\varphi_{|-K_X|}(X)$ is an intersection of quadrics.
\end{enumerate}

Thus, we always can assume that $\varphi_{|-K_X|}$ is an embedding
and $\varphi_{|-K_X|}(X)$ is an intersection of quadrics. Let us
identify $X$ with its anticanonical image $\varphi_{|-K_X|}(X)$.

Let $Z$ be either a smooth point of the threefold $X$, or a
terminal cDV  point (see \cite{YPG}) of the threefold $X$, or a
line in $X\subset\mathbb{P}^{g+1}$ that does not pass through
non-cDV point, or a smooth irreducible conic in
$X\subset\mathbb{P}^{g+1}$ that does not pass through non-cDV
point. Let $\alpha\colon \widetilde{X}\to X$ be a blow-up of the ideal sheaf of the
subvariety $Z\subset X$.

\begin{lem}
\label{lemma:Fanos-basic-links-1} Suppose that $Z$ is either a cDV point or a line.
Then $|-K_{\widetilde{X}}|$ is free from base points.
\end{lem}

\begin{proof}
This follows from an assumption that $\varphi_{|-K_X|}(X)$ is an embedding.
\qed\end{proof}

If $Z$ is a smooth point, let $\beta\colon X\to X^{\prime}$~be~a
morphism given by $|-K_{\widetilde{X}}|$.

\begin{lem}
\label{lemma:Fanos-basic-links-2} Suppose that $Z$ is either a cDV point or a line. Then the morphism $\beta$ is birational
and $X^{\prime}$ is a Fano variety with canonical Gorenstein
singularities such that $-K_{X^{\prime}}$ is very ample.
\end{lem}

\begin{proof}
The required assertion follows from the fact that $X$ is an
intersection of quadrics.
\qed\end{proof}

If $Z$ is a conic, then we need to impose few additional
assumption on $X$ and $Z$ (cf. \cite[Theorem~1.8]{Takeuchi}) to be
sure that the morphism $\beta$ is birational, and $X^{\prime}$ is
a Fano variety with canonical Gorenstein singularities such that
$-K_{X^{\prime}}$ is very ample. In toric case these conditions
can be easily verified.

Let $\pi\colon X\dasharrow X^{\prime}$ be a projection from $Z$.
If $Z$ is not a smooth point, then the diagram
\begin{equation}
\label{equation:basic-link-Fanos}
\xymatrix{&\widetilde{X}\ar@{->}[dl]_{\alpha}\ar@{->}[dr]^{\beta}&\\%
X\ar@{-->}[rr]_{\pi}&&X^{\prime}}
\end{equation}
commutes. Unfortunately, if $Z$ is a smooth point, then the
diagram~(\ref{equation:basic-link-Fanos}) does not commute. In this case, we should define the \emph{basic link}
between Fano threefolds in a slightly different way. Namely, if
$Z$ is a smooth point, we still can consider the commutative
diagram~(\ref{equation:basic-link-Fanos}), but we have to assume
that $\pi$ is a projection from the projective tangent space to
$X$ at the point $Z$ (instead of projection from $Z$ like in other
cases). Moreover, if $Z$ is a smooth point, similar to the case
when $Z$ is a conic, we must impose few additional assumptions on
$X$ and $Z$ to be sure that the morphism $\beta$ is birational,
and $X^{\prime}$ is a Fano variety with canonical Gorenstein
singularities such that $-K_{X^{\prime}}$ is very ample. These
conditions can be easily verified in many cases --- in the toric
case, in the case of index bigger then $1$, 
see Remark~\ref{remark:index-2-case}.

We are going call the diagram~(\ref{equation:basic-link-Fanos}) a
\emph{basic link} between Fano threefolds of type
\begin{itemize}
\item \SarSP\ if $Z$ is a smooth point,
\item \SarDP\  (or \SarODP\ or \SarCDV, respectively) if $Z$ is a double point (ordinary double point or non-ordinary double point, respectively),%
\item \SarL\ if $Z$ is a line,%
\item \SarC\ if $Z$ is a conic.
\end{itemize}
Moreover, in all possible cases, we are going to call $Z$ a
\emph{center} of the basic link~(\ref{equation:basic-link-Fanos}) or
\emph{projection center} (of $\pi$).

\begin{oss}
\label{remark:index-2-case} Suppose that $Z$ is a smooth point,
and $-K_{X}\sim 2H$ for some ample Cartier divisor $H$. Put
$d=H^3$. Then the linear system $|H|$ induces a rational map
$\varphi_{|H|}\colon X\dasharrow\mathbb{P}^{d+1}$ (this follows
from the Riemann--Roch Theorem and basic vanishing theorems). If
$\varphi_{|H|}$ is not an embedding, i.e. $H$ is not very ample,
then $X$ can be easily described exactly as in the smooth case
(see \cite{IsPr99}). Namely, one can show that $X$ is either a
hypersurface of degree $6$ in $\mathbb{P}(1,1,1,2,3)$ or a
hypersurface of degree $4$ in $\mathbb{P}(1,1,1,1,2)$. 
Similarly, if $H$ is very ample and $\varphi_{|H|}(X)$ is not an
intersection of quadrics in $\mathbb{P}^{d+1}$, then $X$ is just a
cubic hypersurface in $\mathbb{P}^4$. 
Assuming that
$\varphi_{|H|}(X)$ is an intersection of quadrics in
$\mathbb{P}^{d+1}$ (this is equivalent to $(-K_{X})^{3}>24$) and
identifying $X$ with its image $\varphi_{|H|}(X)$ in
$\mathbb{P}^{d+1}$, we see that there exists a commutative diagram
\begin{equation}
\label{equation:basic-link-Fanos-point}
\xymatrix{&\widetilde{X}\ar@{->}[dl]_{\alpha}\ar@{->}[dr]^{\beta}&\\%
X\ar@{-->}[rr]_{\pi}&&X^{\prime},}
\end{equation}
where $\pi\colon X\dasharrow X^{\prime}$ is a projection of the
threefold $X\subset\mathbb{P}^{d+1}$ from the point $Z$. Then
$X^{\prime}$ is a Fano threefold with canonical Gorenstein
singularities whose Fano index is divisible by $2$ as well. 
\end{oss}

Similar to the two-dimensional case, we can take $\mathbb{P}^3$ or
an irreducible quadric in $\mathbb{P}^3$ and start applying basic
links interactively. Hypothetically, this would give us all (or
almost all) deformation types of Fano threefolds with canonical
Gorenstein singularities whose anticanonical degree is at most
$64$ (the anticanonical degree decrees after the basic link).

\subsection{Toric basic links between toric Fano threefolds with canonical Gorenstein singularities.}
\label{subsection:3dim basic links}
Let $X$ be a toric Fano threefold with canonical Gorenstein
singularities. Let us fix the action of the torus
$(\mathbb{C}^{*})^3$ on $X$. Suppose that $-K_{X}$ is very ample
and $X$ is not trigonal. Then we can identify $X$ with its
anticanonical image in $\mathbb{P}^{g+1}$, where $g=(-K_X)^3/2+1$
(usually called the genus of the Fano threefold $X$). If $Z$ is
not a smooth point of the threefold $X$, then the commutative
diagram~(\ref{equation:basic-link-Fanos}) is torus invariant as
well, and we call the basic link~\ref{equation:basic-link-Fanos} a
\emph{toric basic link}. This gives us three types of toric basic
links:  \SarDP\ if $Z$ is a double point (\SarODP\ if $Z$ is an ordinary double point, and \SarCDV\ if $Z$ is non-ordinary double point), \SarL\ if $Z$ is a line, and \SarC\ if $Z$ is a
conic. In the case when $Z$ is a smooth torus invariant point, we
proceed as in Remark~\ref{remark:index-2-case} and obtain the
toric basic link of type \SarSP\ assuming that the
Fano index of the threefold $X$ is divisible by $2$ or $3$ and
$(-K_{X})^{3}>24$.

We can take $X=\mathbb{P}^3$ and start applying toric basic links
until we get a toric Fano threefold with canonical Gorenstein
singularities to whom we can not apply any toric basic link (e.g.
when we get a toric quartic hypersurface in $\mathbb{P}^{4}$).
Hypothetically, this would give us birational maps between almost all toric Fano threefolds with canonical
Gorenstein singularities whose anticanonical degree is at most
$64$. Similarly, we can take into account irreducible quadrics in
$\mathbb{P}^4$ to make our picture look more complicated and,
perhaps, refined. Moreover, we can start with
$X=\mathbb{P}(1,1,1,3)$ or $X=\mathbb{P}(1,1,4,6)$, which are the
highest anticanonical degree Fano threefolds with canonical
Gorenstein singularities (see \cite{Pr05}) to get possibly all
toric Fano threefolds with canonical Gorenstein singularities.
Keeping in mind that there are $4319$ such toric Fano threefolds,
we see that this problem requires a lot of computational efforts
and usage of databases of toric Fano threefolds (see \cite{Br}).

Let us restrict our attention to toric Fano threefolds with
canonical Gorenstein singularities that are known to be smoothable
to smooth Fano threefolds with Picard group $\mathbb{Z}$. Then
starting with $\mathbb{P}^3$ and 
with singular
quadric in
$\mathbb{P}^4$ with one ordinary double point 
and taking into account \emph{some} toric basic
links, we obtain Figure~\ref{figure:Fano snake},
\begin{figure}[h]
$$
\hspace{-4pc}
\xymatrix{
 & & {\mathbf Q} \ar@{-->}[d]^-{\mathrm {II_{p}}} & & & & & & \\
{\mathbf \PP^3} \ar@{-->}[d]^-{\mathrm {II_{p}}} & & X_{2.30} \ar@{-->}[d]^-{\mathrm {II_{l}}} & & & & & & \\
V_{2.35} \ar@{-->}[d]^-{\mathrm {II_{p}}} & & X_{3.23} \ar@{-->}[d]^-{\mathrm {II_{odp}}} & & & & & & \\
V_{2.32} \ar@{-->}[d]^-{\mathrm {II_{p}}} \ar@{-->}[r]^-{\mathrm {II_c}} &
X_{3.24} \ar@{-->}[r]^-{\mathrm {II_{odp}}} & X_{4.9}
\ar@{-->}[r]^-{\mathrm {II_c}} & X_{4.6} \ar@{-->}[r]^-{\mathrm {II_c}} &
X_{3.12} \ar@{-->}[r]^-{\mathrm {II_{odp}}} & X_{3.10} \ar@{-->}[r]^-{\mathrm {II_{odp}}} & X_{4.1} \ar@{-->}[r]^-{\mathrm {II_{odp}}} & {\mathbf X_{22}} \ar@{-->}[r]^-{\mathrm {II_{odp}}} & X_{2.13} \ar@{-->}[d]^-{\mathrm {II_{odp}}} \\
{\mathbf V_5} \ar@{-->}[d]^-{\mathrm {II_{p}}} & & & & & & & & {\mathbf X_{18}} \ar@{-->}[d]^-{{\mathrm {II_{cDV}}}}\\
{\mathbf V_4} \ar@{-->}[d]^-{\mathrm {II_{p}}} & & {\mathbf X_4} & {\mathbf X_6} \ar@{-->}[l]_-{{\mathrm {II_{cDV}}}} & {\mathbf X_8} \ar@{-->}[l]_-{{\mathrm {II_{cDV}}}} & {\mathbf X_{10}} \ar@{-->}[l]_-{{\mathrm {II_{cDV}}}} & {\mathbf X_{12}} \ar@{-->}[l]_-{\mathrm {II_{odp}}} & {\mathbf X_{14}} \ar[l]_-{\mathrm {II_{odp}}} & {\mathbf X_{16}} \ar@{-->}[l]_-{\mathrm {II_{odp}}} \\
{\mathbf V_3} \ar@{-->}[d]^-{\mathrm {II_{p}}} & & & & & & & & \\
{\mathbf V_2} & & }$$
\caption{Fano snake.}\label{figure:Fano
snake}
\end{figure}
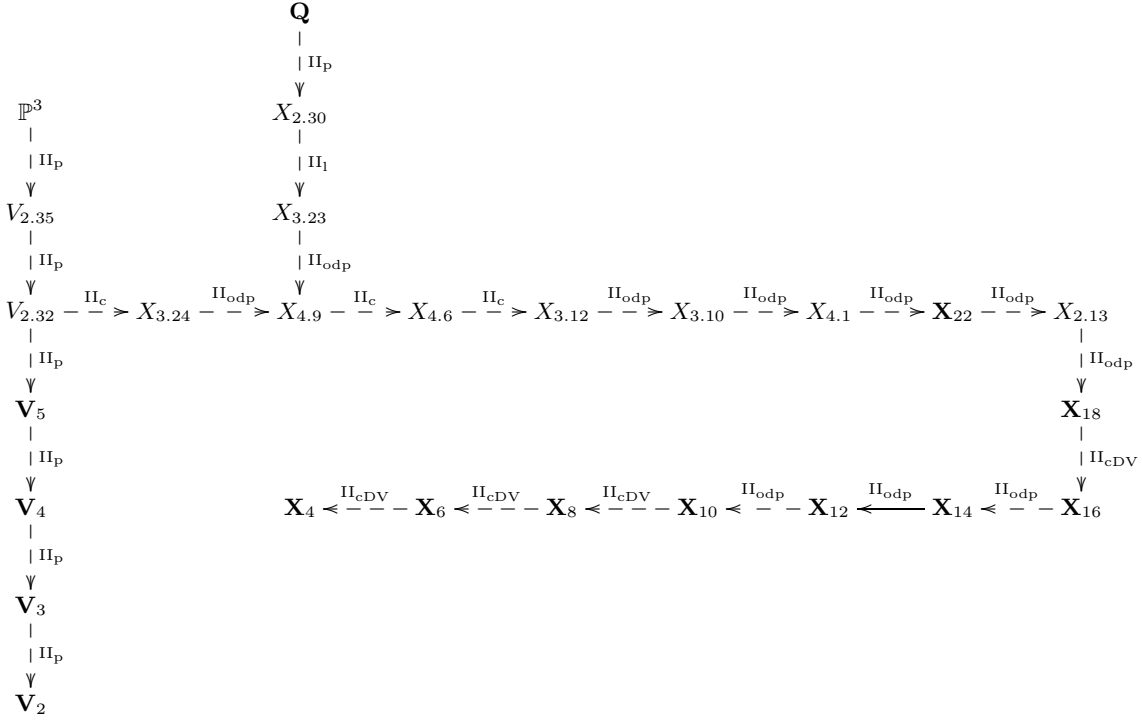
where we use bold fonts to denote Fano threefolds with Picard
group $\mathbb{Z}$. 

Recent progress in Mirror Symmetry for smooth Fano threefolds (see
\cite{Golyshev2002}, \cite{Golyshev2007}, \cite{ALESS}, \cite{Prz07},
\cite{Prz09}, \cite{fanosearch}) shed a new light on and attracts
a lot of attention to toric degenerations of smooth Fano
threefolds (see \cite{Batyrev1994}, \cite{Batyrev2004}, \cite{Ga},
\cite{ILP11}, \cite{CI12}, \cite{Batyrev2012}). It would be
interesting to see the relation between toric basic links between
smoothable toric Fano threefolds with canonical Gorenstein
singularities, basic links between smoothable Fano threefolds with
canonical Gorenstein singularities, their toric degenerations, and
geometry of their Landau--Ginzburg models (cf. \cite{Prz07}).


\begin{prop}[{\cite[Theorem~2.8]{IV09}}]
\label{proposition:Ilten}
Consider a Laurent polynomial
$p_1=xg_1g_2+g_3+g_4/x$, where $g_i$ are Laurent polynomials that do not depend on $x$.
Let $p_2=xg_1+g_3+g_2g_4/x$.
Let $T_i$ be a toric variety whose fan polytope is a Newton polytope of $p_i$.
Then $T_2$ deforms to $T_1$.
\end{prop}

\begin{oss}
\label{remark:2 smoothings}
In~\cite[Example 2.3]{JR11}, Jahnke and Radloff considered an anticanonical cone over del Pezzo surface $S_6$
(the rightmost on 4th line of Figure~\ref{figure:del Pezzo tree}) of degree 6 and showed
that it has two smoothings, to $X_{2.32}$ and $X_{3.27}$. Notice that $S_6$ has 4 canonical toric degenerations; all of them are projections from $\PP^2$ and two of them, $S_6$ and $S_6''$ (the third on 4th line of Figure~\ref{figure:del Pezzo tree}), are projections of a smooth quadric surface. Cones over these varieties have fan polytopes numbers 155 and 121
(according to~\cite{fanosearch}) correspondingly. Two these polytopes are exactly ones having two Minkowski decompositions each of which
gives constant terms series for $X_{2.32}$ and $X_{3.27}$. So we have two smoothings corresponding to two pairs of Minkowski decomposition.
The question is why the existence of two deformations to two different varieties corresponds to the fact that the toric varieties are
projections from quadric surface.
\end{oss}

\begin{ex}
Consider a Laurent polynomial $$p_1=xy+xz+xyz+x/y+x/z+x+1/x.$$ One can prove that it is a toric Landau--Ginzburg model for $X_{2.35}$. Indeed,
one can directly check period and Calabi--Yau conditions. To prove toric condition one can observe that
$$p_1=x(z+z/y+1)(y+1/z)+1/x.$$ So, by Proposition~\ref{proposition:Ilten}, toric variety $T_{p_1}$ associated with $p_1$ can be deformed to
toric variety associated with $$p_2=x(z+z/y+1)+(y+1/z)/x$$ which after toric change of variables coincides with $f_{2}$. We get $F_2$ which
can be smoothed to $X_{2.35}$ by Theorem~\ref{theorem:toric table}.

A variety $T_{p_1}$ is nothing but a cone over a toric del Pezzo surface $S_7$ (the rightmost in the 3rd column of Figure~\ref{figure:del Pezzo tree}).
Consider a basic link --- projection from a smooth point on $T_{p_1}$. One get a toric variety --- cone over del Pezzo surface
$S_6$. It has two smoothings (see Remark~\ref{remark:2 smoothings}). Moreover, there are two toric Landau--Ginzburg models,
$$p_3=xy+xz+xyz+x/y+x/z+x/y/z+2x+1/x$$ and $$p_4=xy+xz+xyz+x/y+x/z+x/y/z+3x+1/x,$$ one for $X_{2.32}$ and another one for $X_{3.27}$.
Indeed, as before period and Calabi--Yau conditions can be checked directly. Notice that
$$p_3=x\left((yz+1)/z/y\right)(y+1)(z+1)+1/x$$ and $$p_4=x\left((yz+z+1)/y/z\right)(yz+y+1)+1/x.$$ After cluster change of variables
by Proposition~\ref{proposition:Ilten} one get two polynomials
associated with $F_3$ and $F_4$. These varieties by
Theorem~\ref{theorem:toric table} can be smoothed to $X_{2.32}$
and $X_{3.27}$. Nevertheless, the correct target for basic link
between the last two varieties is $X_{2.32}$, since projecting a
general $X_{2.32}$ from a point, we always obtain $X_{3.27}$.
\end{ex}

\begin{oss}
The same can be done with another variety, cone over $S_6''$ (the third in the forth column of Figure~\ref{figure:del Pezzo tree}).
Indeed, by Proposition~\ref{proposition:Ilten} a variety $S_6''$
is a degeneration of $S_6$ so cones over them degenerates as
well.
\end{oss}




\section{Classical theory of Landau--Ginzburg models}
\label{subsubsection:MSVHS}
From now on we concentrate on the theory of Landau--Ginzburg models and their moduli.
First we recall the classical definition of Landau--Ginzburg model of a single Fano variety.
More precisely see, say,~\cite{Prz07} and references therein.
Let us have a smooth Fano variety $X$ of dimension $n$. We can associate \emph{a quantum cohomology ring} $QH^*(X)=H^*(X,\QQ)\otimes_\ZZ \Lambda$ to it, where $\Lambda$
is the Novikov ring for $X$. The multiplication in this ring, the so called \emph{quantum multiplication}, is given by \emph{(genus zero)
Gromov--Witten invariants} --- numbers counting rational curves lying on it. Given this data one can associate \emph{a regularized quantum differential operator} $Q_X$ (the second Dubrovin's connection) --- the regularization of an operator associated with connection in the trivial vector bundle
given by quantum multiplication by the canonical class $K_X$.
Solutions of an equation given by this operator are given by \emph{$I$-series} for $X$ --- generating series for its one-pointed Gromov--Witten invariants.
In particular, one ``distinguished'' solution is a constant term (with respect to cohomology) of $I$-series.
Let us denote it by
$I=1+a_1t+a_2t^2+\ldots$.

\begin{definition}
\label{definition: toric LG}
\emph{A toric Landau--Ginzburg model} is a Laurent polynomial $f\in \CC[x_1^{\pm 1}, \ldots, x_n^{\pm}]$ such that:
\begin{description}
  \item[Period condition] The constant term of $f^i\in \CC[x_1^{\pm 1}, \ldots, x_n^{\pm}]$ is $a_i$ for any $i$ (this means that $I$ is a period of a family $f\colon (\CC^*)^n\to \CC$, see~\cite{Prz07}).
  \item[Calabi--Yau condition] There exists a fiberwise compactification (the so called \emph{Calabi--Yau compactification})
  whose total space is a smooth (open) Calabi--Yau variety.
  \item[Toric condition] There is an embedded degeneration $X\rightsquigarrow T$ to a toric variety $T$ whose fan polytope
  (the convex hull of generators of its rays) coincides with the Newton polytope (the convex hull of non-zero coefficients) of $f$.
\end{description}
\end{definition}

\begin{oss}
This notion can be extended to some non-smooth cases, see, for instance,~\cite{ALESS}.
\end{oss}

\begin{teo}[{\cite[Theorem 18]{Prz09} and~\cite[Theorem 3.1]{ILP11}}]
\label{theorem:rank 1 threefolds}
Smooth Fano threefolds of Picard rank 1 have toric Landau--Ginzburg models.
\end{teo}

\begin{oss}
Toric Landau--Ginzburg models for Picard rank 1 Fano threefolds are found in~\cite{Prz09}. However they are not unique.
Some of them coincide with ones from Table~\ref{table}. Anyway Theorem~\ref{theorem:rank 1 threefolds} holds for all threefolds
from Table~\ref{table}, see Theorem~\ref{theorem:toric table}.
\end{oss}


%
%

\begin{teo}[\cite{DKLP}]
\label{theorem:Shioda-Inose}
Let $X$ be a Fano threefold of index $i$ and $(-K_X)^3=i^3k$. Then fibers of toric weak Landau--Ginzburg model from~\cite{Prz09}
can be compactified to Shioda--Inose surfaces with Picard lattice  $H\oplus E_8(-1)\oplus E_8(-1)+\langle -ik \rangle$.
\end{teo}

\begin{oss}
This theorem holds for toric Landau--Ginzburg models for Fano threefolds of Picard rank 1 from Table~\ref{table}.
\end{oss}

This theorem means that fibers of compactified toric Landau--Ginzburg models are mirrors of anticanonical sections
of corresponding Fano varieties, and this property determines compactified toric Landau--Ginzburg models uniquely
as the moduli spaces of possible mirror K3's are just $\PP^1$'s.


So all above can be summarized to the following Mirror Symmetry conjecture.

\begin{conj}
\label{conjecture:MS}
Every smooth Fano variety has a toric Landau--Ginzburg model.
\end{conj}


Theorem~\ref{theorem:rank 1 threefolds} 
shows that this conjecture holds for Fano threefolds of Picard rank 1.
Theorem~\ref{theorem:toric table} shows that the conjecture holds for Fano varieties from Table~\ref{table}.

\subsection{The table}

Now we study toric Landau--Ginzburg models for Fano threefolds of Picard rank 1 given by toric basic links from $\PP^3$
and quadric. At first we give a table of such toric Landau--Ginzburg models and make some remarks concerning it. Then we prove (Theorem~\ref{theorem:toric table}) that Laurent polynomials listed in the table
are toric Landau--Ginzburg models of Fano threefolds.

Table~\ref{table} is organized as follows. $N$ is a number of a
variety in the table. Var. is a Fano smoothing numerated
following~\cite{IsPr99}. Deg. is a degree of a variety. Par. is a
number of variety(ies) giving our variety by a projection. BL is a
type of toric basic link(s). 
Desc. stands for descendants --- varieties that can be obtained by projection
from given variety.
The last column is a toric Landau--Ginzburg model for the variety.

\begin{scriptsize}
\begin{longtable}{||c|c|c|c|c|c|c|p{7cm}||}
  \hline
  $N$ & $\mathrm{Var.}$ & $\mathrm{Deg}.$ & $\mathrm{Par.}$ & $\mathrm{BL}$ & $\mathrm{Desc.}$
  &
\begin{minipage}[c]{7cm}
\centering
\vspace{.1cm}

  Weak LG model

\vspace{.1cm}

\end{minipage}
\\
  \hline
  \hline
  1 & 1.17 & $2^3\cdot 8$ & $\emptyset$ & $\emptyset$ & 2
&
\begin{minipage}[c]{7cm}
\medskip

$x+y+z+\frac{1}{xyz}$

\medskip
\end{minipage}
    \\
  \hline
  2 & 2.35 & $2^3\cdot 7$ & $1$ & \SarSP & 3
&
\begin{minipage}[c]{7cm}
\medskip

$x+y+z+\frac{1}{xyz}+\frac{1}{x}$

\medskip
\end{minipage}
    \\
  \hline
  3 & 2.32 & $2^3\cdot 6$ & $2$ & \SarSP & 5, 9
&
\begin{minipage}[c]{7cm}
\medskip

$x+y+z+\frac{1}{xyz}+\frac{1}{x}+\frac{1}{y}$

\medskip
\end{minipage}
    \\
  \hline
  4 & 3.27 & $2^3\cdot 6$ & $\emptyset$ & $\emptyset$ & 5
&
\begin{minipage}[c]{7cm}
\medskip

$x+y+z+\frac{1}{x}+\frac{1}{y}+\frac{1}{z}$

\medskip
\end{minipage}
    \\
  \hline
  5 & 1.15 & $2^3\cdot 5$ & 3, 4 & \SarSP, \SarSP & 6
&
\begin{minipage}[c]{7cm}
\medskip

$x+y+z+\frac{1}{xyz}+\frac{1}{x}+\frac{1}{y}+\frac{1}{z}$

\medskip
\end{minipage}
    \\
  \hline
  6 & 1.14 & $2^3\cdot 4$ & 5 & \SarSP &  7
&
\begin{minipage}[c]{7cm}
\medskip

$x+y+z+\frac{1}{xyz}+\frac{2}{x}+\frac{1}{y}+\frac{1}{z}+\frac{yz}{x}$

\medskip
\end{minipage}
    \\
  \hline
  7 & 1.13 & $2^3\cdot 3$ & 6 & \SarSP & 8
&
\begin{minipage}[c]{7cm}
\medskip

$x+y+2z+\frac{1}{xyz}+\frac{2}{x}+\frac{2}{y}+\frac{1}{z}+\frac{yz}{x}+\frac{xz}{y}$

\medskip
\end{minipage}
    \\
  \hline
  8 & 1.12 & $2^3\cdot 2$ & 7 & \SarSP & $\emptyset$  
&
\begin{minipage}[c]{7cm}
\medskip

$2x+2y+2z+\frac{1}{xyz}+\frac{2}{x}+\frac{2}{y}+\frac{2}{z}+\frac{yz}{x}+\frac{xz}{y}+\frac{xy}{z}$

\medskip
\end{minipage}
    \\
  \hline
  9 & 3.24 & 42 & 3 & \SarC & 10
&
\begin{minipage}[c]{7cm}
\medskip

$x+y+z+\frac{1}{xyz}+\frac{1}{x}+\frac{1}{y}+\frac{1}{yz}$

\medskip
\end{minipage}
    \\
  \hline
  10 & 4.9 & 40 & 9, 27 & \SarODP, \SarODP & 11
&
\begin{minipage}[c]{7cm}
\medskip

$x+y+z+\frac{1}{xyz}+\frac{1}{x}+\frac{1}{y}+\frac{1}{yz}+\frac{1}{xy}$

\medskip
\end{minipage}
    \\
  \hline
  11 & 4.6 & 34 & 10 & \SarC & 12
&
\begin{minipage}[c]{7cm}
\medskip

$x+y+z+\frac{1}{xyz}+\frac{1}{x}+\frac{1}{y}+\frac{1}{yz}+\frac{1}{xy}+yz$

\medskip
\end{minipage}
    \\
  \hline
  12 & 3.12 & 28 & 11 & \SarC & 13
&
\begin{minipage}[c]{7cm}
\medskip

$x+y+z+\frac{1}{xyz}+\frac{1}{x}+\frac{1}{y}+\frac{1}{yz}+\frac{1}{xy}+yz+xy$

\medskip
\end{minipage}
    \\
  \hline
  13 & 3.10 & 26 & 12 & \SarODP & 14
&
\begin{minipage}[c]{7cm}
\medskip

$x+y+z+\frac{1}{xyz}+\frac{1}{x}+\frac{1}{y}+\frac{1}{yz}+\frac{1}{xy}+yz+xy+\frac{1}{z}$

\medskip
\end{minipage}
    \\
  \hline
  14 & 4.1 & 24 & 13 & \SarODP\footnote{We make toric change of variables $\frac{y}{x}\to y$, $\frac{z}{y}\to z$.} & 15
&
\begin{minipage}[c]{7cm}
\medskip

$x+y+z+\frac{1}{x}+\frac{1}{y}+\frac{1}{z}+\frac{x}{y}+\frac{x}{z}+\frac{y}{x}+\frac{y}{z}+\frac{z}{x}+\frac{z}{y}$

\medskip
\end{minipage}
    \\
  \hline
  15 & 1.10 & 22 & 14 & \SarODP & 16
&
\begin{minipage}[c]{7cm}
\medskip

$x+y+z+\frac{1}{x}+\frac{1}{y}+\frac{1}{z}+\frac{x}{y}+\frac{x}{z}+\frac{y}{x}+\frac{y}{z}+\frac{z}{x}+\frac{z}{y}+\frac{x}{yz}$

\medskip
\end{minipage}
%
%
    \\
  \hline
  16 & 2.13 & 20 & 15 & \SarODP & 17
&
\begin{minipage}[c]{7cm}
\medskip

$x+y+z+\frac{1}{x}+\frac{1}{y}+\frac{2}{z}+\frac{x}{y}+\frac{x}{z}+\frac{y}{x}+\frac{y}{z}+\frac{z}{x}+\frac{z}{y}+\frac{x}{yz}+\frac{y}{xz}$

\medskip
\end{minipage}
    \\
  \hline
  17 & 1.9 & 18 & 16 & \SarODP & 18
&
\begin{minipage}[c]{7cm}
\medskip

$x+y+z+\frac{2}{x}+\frac{2}{y}+\frac{2}{z}+\frac{x}{y}+\frac{x}{z}+\frac{y}{x}+\frac{y}{z}+\frac{z}{x}+\frac{z}{y}+\frac{x}{yz}$
$+\frac{y}{xz}+\frac{z}{xy}$

\medskip
\end{minipage}
    \\
  \hline
  18 & 1.8 & 16 & 17 & \SarCDV & 19
&
\begin{minipage}[c]{7cm}
\medskip

$x+y+z+\frac{3}{x}+\frac{3}{y}+\frac{3}{z}+\frac{x}{y}+\frac{x}{z}+\frac{y}{x}+\frac{y}{z}+\frac{z}{x}+\frac{z}{y}+\frac{x}{yz}$
$+\frac{y}{xz}+\frac{z}{xy}+\frac{1}{xyz}+\frac{2}{xy}+\frac{2}{yz}+\frac{2}{xz}$

\medskip
\end{minipage}
    \\
  \hline
  19 & 1.7 & 14 & 18 & \SarODP & 20
&
\begin{minipage}[c]{7cm}
\medskip

$x+y+z+\frac{3}{x}+\frac{3}{y}+\frac{4}{z}+\frac{x}{y}+\frac{2x}{z}+\frac{y}{x}+\frac{2y}{z}+\frac{z}{x}+\frac{z}{y}+\frac{x}{yz}$
$+\frac{y}{xz}+\frac{z}{xy}+\frac{1}{xyz}+\frac{2}{xy}+\frac{2}{yz}+\frac{2}{xz}+\frac{xy}{z}$

\medskip
\end{minipage}
    \\
  \hline
  20 & 1.6 & 12 & 19 & \SarODP & 21
&
\begin{minipage}[c]{7cm}
\medskip

$2x+y+z+\frac{3}{x}+\frac{4}{y}+\frac{4}{z}+\frac{2x}{y}+\frac{2x}{z}+\frac{y}{x}+\frac{2y}{z}+\frac{z}{x}+\frac{2z}{y}+\frac{x}{yz}$
$+\frac{y}{xz}+\frac{z}{xy}+\frac{1}{xyz}+\frac{2}{xy}+\frac{2}{yz}+\frac{2}{xz}+\frac{xy}{z}+\frac{xz}{y}$

\medskip
\end{minipage}
    \\
  \hline
  21 & 1.5 & 10 & 20 & \SarODP & 22
&
\begin{minipage}[c]{7cm}
\medskip

$2x+2y+2z+\frac{4}{x}+\frac{4}{y}+\frac{4}{z}+\frac{2x}{y}+\frac{2x}{z}+\frac{2y}{x}+\frac{2y}{z}+\frac{2z}{x}+\frac{2z}{y}$
$+\frac{x}{yz}+\frac{y}{xz}+\frac{z}{xy}+\frac{1}{xyz}+\frac{2}{xy}+\frac{2}{yz}+\frac{2}{xz}+\frac{xy}{z}+\frac{xz}{y}+\frac{yz}{x}$

\medskip
\end{minipage}
    \\
  \hline
  22 & 1.4 & 8 & 21 & \SarCDV & 23
&
\begin{minipage}[c]{7cm}
\medskip

$4x+4y+4z+\frac{4}{x}+\frac{4}{y}+\frac{4}{z}+\frac{2x}{y}+\frac{2x}{z}+\frac{2y}{x}+\frac{2y}{z}+\frac{2z}{x}+\frac{2z}{y}$
$+\frac{x}{yz}+\frac{y}{xz}+\frac{z}{xy}+\frac{1}{xyz}+\frac{2}{xy}+\frac{2}{yz}+\frac{2}{xz}+\frac{xy}{z}+\frac{xz}{y}+\frac{yz}{x}$
$+xyz+2xy+2xz+2yz$

\medskip
\end{minipage}
    \\
  \hline
  23 & 1.3 & 6 & 22 & \SarCDV\footnote{Under toric change of variables.} & 24
&
\begin{minipage}[c]{7cm}
\medskip

$\frac{(y+z+1)^2((y+z+1)^2+2x(y+z+1)+x^2)}{xyz}-12$

\medskip
\end{minipage}
    \\
  \hline
  24 & 1.2 & 4 & 23 & \SarCDV & $\emptyset$
&
\begin{minipage}[c]{7cm}
\medskip

$\frac{(x+y+z+1)^4}{xyz}$

\medskip
\end{minipage}
    \\
  \hline
%
%
  25 & 1.16 & 54 & $\emptyset$ & $\emptyset$ & 26
&
\begin{minipage}[c]{7cm}
\medskip

$x+y+z+\frac{1}{xy}+\frac{1}{yz}$

\medskip
\end{minipage}
    \\
  \hline
  26 & 2.30 & 46 & 27 & \SarSP & 27
&
\begin{minipage}[c]{7cm}
\medskip

$x+y+z+\frac{1}{xy}+\frac{1}{yz}+\frac{1}{xyz}$

\medskip
\end{minipage}
    \\
  \hline
  27 & 3.23 & 42 & 28 & \SarL & 10
&
\begin{minipage}[c]{7cm}
\medskip

$x+y+z+\frac{1}{xy}+\frac{1}{yz}+\frac{1}{xyz}+\frac{1}{x}$

\medskip
\end{minipage}
\\
  \hline
  \hline
\caption[]{\label{table}{\normalsize Weak
      Landau--Ginzburg models for
    Fano threefolds.}}
\end{longtable}
\end{scriptsize}

\begin{oss}
$F_2$ is a blow-up of $F_1=\PP^3$ at one point with an exceptional divisor $E$. $F_3=X_{2.35}$ is a projection from a point lying far from $E$.
If we project from a point lying on $E$ we get another (singular) variety, $F_3^\prime$, with corresponding weak Landau--Ginzburg model
$$
x+y+z+\frac{1}{xyz}+\frac{2}{x}+\frac{yz}{x}.
$$
\end{oss}

\begin{oss}
The variety {\bf 24}, the toric quartic, has no cDV points or smooth toric lines. So we can not proceed to make basic links.
However it has 4 singular canonical (triple) points and we can project from any of them. In other words, we can project quartic
$$
\{x_1x_2x_3x_4=x_0^4\}\subset \PP[x_0,x_1,x_2,x_3,x_4]
$$
from the point, say, $(0:0:0:0:1)$. Obviously we get the variety {\bf 1}, that is $\PP^3$ again. 
\end{oss}

\begin{prop}
\label{proposition:CY}
Families of hypersurfaces in $(\CC^*)^3$ given by Laurent polynomials from Table~\ref{table} can be fiberwise compactified to (open) Calabi--Yau varieties.
\end{prop}

\begin{proof}
Let $f$ be a Laurent polynomials from the table.
Compactify the corresponding family $\{f=\lambda\}\subset \Spec \CC[x^{\pm 1}, y^{\pm 1}, z^{\pm 1}]\times \Spec \CC[\lambda]$ fiberwise
using standard embedding $\Spec \CC[x^{\pm 1}, y^{\pm 1}, z^{\pm 1}] \subset \Proj \CC[x, y, z, t]$. In other words, multiply it by a denominator
($xyz$) and add an extra homogenous variable $t$. For varieties {\bf 11, 12, 13} do toric change of variables $xy\to x$, $yz\to z$. We get a
family of singular quartics. Thus it has trivial canonical class. Singularities of the threefold we get are du Val along lines and ordinary
double points; the same type of singularities holds after crepant blow-ups of singular lines and small resolutions of ordinary double points.
Thus the threefold admit a crepant resolution; this resolution is a Calabi--Yau compactification we need.
\qed\end{proof}

\begin{prop}
\label{proposition:toric}
Toric varieties from Table~\ref{table} are degenerations of corresponding Fano varieties.
\end{prop}

\begin{proof}
Varieties {\bf 1--5, 9--15, 25--27} are terminal Gorenstein toric Fano threefolds. So, by~\cite{Na97} they can be smoothed.
By~\cite[Corollary~3.27]{Ga} smoothings are Fanos with the same numerical invariants as the initial toric varieties.
The only smooth Fano threefolds with given invariants are listed at the second column. In other words,
the statement of the proposition for varieties {\bf 1--5, 9--15, 26--28} follows from the proof of~\cite[Theorem~2.7]{Ga}.

Varieties {\bf 6--8} are complete intersections in (weighted) projective spaces. We can write down the dual polytopes
to their fan polytopes. The equations of the toric varieties correspond to homogenous relations on integral points
of the dual polytopes. 
One can see that the relations are binomials
defining corresponding complete intersections and the equations of a Veronese map $v_2$. So the toric varieties can be smoothed to
the corresponding complete intersections. For more details see~\cite[Theorem~2.2]{ILP11}.

Variety $X_{2.13}$ can be described as a section of $\PP^2\times \PP^4$ by divisors of type $(1,1)$, $(1,1)$, and $(0,2)$ (see, say,~\cite{fanosearch}).
Equations of $\PP^2\times \PP^4$ in Segre embedding can be described as all $(2\times 2)$-minors of a matrix
$$
\left(
  \begin{array}{ccccc}
    x_{00} & x_{01} & x_{02} & x_{03} & x_{04} \\
    x_{10} & x_{11} & x_{12} & x_{13} & x_{14} \\
    x_{20} & x_{21} & x_{22} & x_{23} & x_{24} \\
  \end{array}
\right).
$$
Consider its section $T$ given by equations $x_{00}=x_{11}$, $x_{11}=x_{22}$, $x_{01}x_{02}=x_{03}x_{04}$. These equations
give divisors of types $(1,1)$, $(1,1)$, and $(0,2)$ respectively. They are binomial, which means that $T$ is a toric variety.
The equations giving variety {\bf 16} are homogenous integral relations on integral points of a polytope dual to
a Newton polytope of $f_{16}$. 
It is easy to see that these relations are exactly ones defining $T$. Thus $T=F_{16}$ and $F_{16}$ can be smoothed to $X_{2.13}$.

Varieties {\bf 17} and {\bf 18} correspond to ones from~\cite{Prz09}. Thus, by~\cite[Theorem~3.1]{ILP11} they can be smoothed to corresponding Fano threefolds.

The dual polytope to the fan polytope for variety {\bf 19} is drawn on Figure~\ref{figure:V14}.

\begin{figure}[htb]
\centering
\includegraphics[width=5cm]{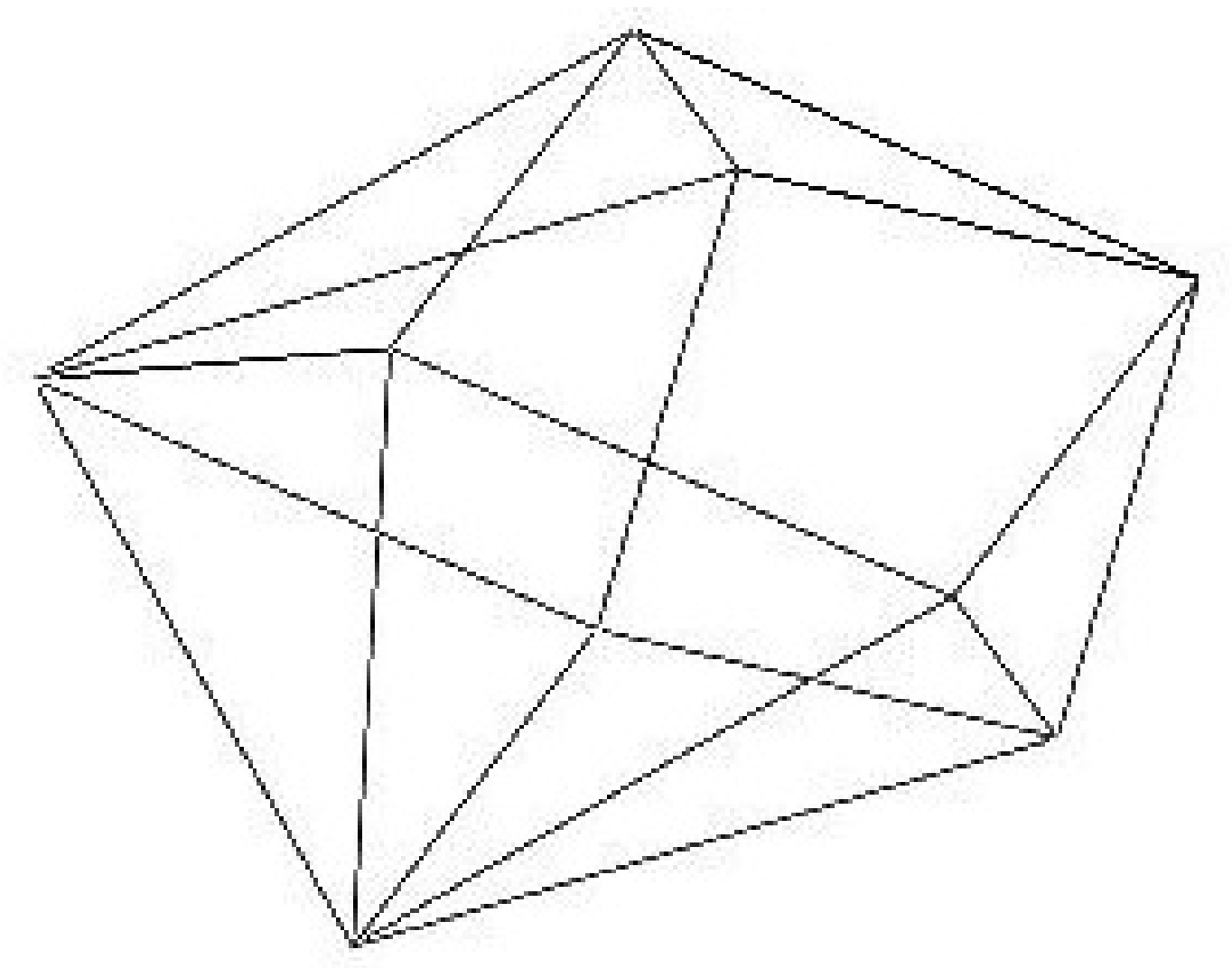}
\caption{Polytope defining variety {\bf 19}.}\label{figure:V14}
\end{figure}

It obviously has a triangulation on 14 triangles satisfying conditions of~\cite[Corollary~3.4]{CI11}. By this corollary the variety~{\bf 19} can be
smoothed to the variety we need.

Finally the existence of smoothings of varieties~{\bf 20--24} to corresponding Fano varieties follows from~\cite{CI12}.
\qed\end{proof}

\begin{teo}
\label{theorem:toric table}
Laurent polynomials from Table~\ref{table} are toric Landau--Ginzburg models for corresponding Fano threefolds.
\end{teo}

\begin{proof}
Period condition follows from direct computations, see~\cite{fanosearch}. Calabi--Yau condition holds by Proposition~\ref{proposition:CY}.
Toric condition holds by Proposition~\ref{proposition:toric}.
\qed\end{proof}

\section{Landau--Ginzburg considerations}
\label{section:discussion}

\subsection{Categorical background}

The examples in the  previous two sections suggest a new approach to birational  geometry of Fano manifolds. This approach amounts to studying all Fano manifolds together.  In this section we summarize this approach  and  give technical tolls for using it. We proceed by extending Kawamata's approach described in Table~\ref{tab:CPINTRO}. We add additional data to categorical approach recorded in the geometry of the  moduli space of Landau--Ginzburg models. The  main points are:
\begin{enumerate}
  \item There exists a  moduli space of Landau--Ginzburg models for many (possibly all)  3-dimensional Fano manifolds.
  \item
The topology of this compactified moduli space of  Landau--Ginzburg models
determines Sarkisov links  among these Fano manifolds. In fact we conjecture that geometry of  moduli space of Landau--Ginzburg models gives answers to many questions related to rationality and birational equivalence --- we suggest some invariants and give examples.
\end{enumerate}

The geometry of the  moduli space of Landau--Ginzburg models was introduced in
\cite{KKP}, \cite{DKK1}, and \cite{DKK2} as analogy with Nonabelian Hodge theory. We  describe this analogy.
We build the ``twistor'' family so that the fiber over zero is the ``moduli space''  of Landau--Ginzburg models and the generic fiber is the Stability Hodge Structure (see below).

Non-commutative Hodge theory endows the cohomology groups of a dg-category with additional linear data --- the non-commutative Hodge structure --- which
records important information about the geometry of the
category. However, due to their linear nature, non-commutative Hodge
structures are not sophisticated enough to codify the full
geometric information hidden in a dg-category. In view of the
homological complexity of such categories it is clear that only a
subtler non-linear Hodge theoretic entity can adequately capture the
salient features of such categorical or non-commutative geometries. In this section by analogy with ``classical nonabelian Hodge theory''  we  construct and  study from such an prospective a new type of entity of exactly such
type --- the Stability Hodge Structure (SHS) associated with a dg-category.

As the name suggests, the SHS of a category is related to the
Bridgeland stabilities on this category.  The moduli space ${\sf
  Stab}_{C}$ of stability conditions of a triangulated dg-category
$C$, is in general, a complicated curved space, possibly with
fractal boundary. In the special case when $C$ is the Fukaya
category of a Calabi--Yau threefold, the space ${\sf Stab}_{C}$ admits
a natural one-parameter specialization to a much simpler space ${\sf
  S}_{0}$. Indeed, HMS predicts that the moduli space of complex
structures on the mirror Calabi--Yau threefold maps to a Lagrangian
subvariety ${\sf Stab}^{\text{geom}}_{C} \subset {\sf Stab}_{C}$.
The space ${\sf S}_{0}$ is
the fiber at $0$ of this completed family and conjecturally ${\sf S}
\to \mathbb{C}$ is one chart of a twistor-like family $\mathcal{S} \to
\mathbb{P}^{1}$ which is by definition {\em the Stability Hodge
  Structure associated with $C$}.

Stability Hodge Structures are expected to exist for more general dg-categories, in particular for Fukaya--Seidel categories associated with
a superpotential on a Calabi--Yau space or with categories of
representations of quivers.  Moreover, for special non-compact
Calabi--Yau 3-folds, the zero fiber ${\sf S}_{0}$ of a Stability Hodge
Structure can be identified with the Dolbeault realization of a
nonabelian Hodge structure of an algebraic curve. This is an
unexpected and direct connection with Simpson's nonabelian Hodge theory (see~\cite{Si92}) which we exploit further suggesting some geometric applications.

 We briefly recall nonabelian Hodge theory settings. According to Simpson (see~\cite{Si92})
we have one-parametric twistor family such that the fiber over zero is the
moduli space of Higgs bundles and the generic fiber is the moduli space --- $M_{Betti}$ --- of
representation of the fundamental group of over what Higgs bundle is. 
By analogy with the nonabelian Hodge structure we have:

\begin{conj}[\cite{KKP}]
The moduli space of stability conditions of Fukaya--Seidel category can be included in   one parametric ``twistor'' family.
\end{conj}

In other words SHS exists for  Fukaya--Seidel categories. Parts of this conjecture are checked in \cite{KKP} and \cite{HKK}.

%
%
%
%
%
%
%
%
%
%
%
%
%
%
%

We give a brief example of SHS.

\begin{ex}
\label{example:An}
We will give a brief explanation  the calculation of the
``twistor'' family for the SHS for the category  $A_n$ recorded in the picture above.
We start with the moduli space of stability conditions for
the category  $A_n$, which can be identified with
differentials

$$ e^{p(z)}dz, $$
where $p(z)$ is a polynomial of degree $n+1$.

Classical work of Nevanlinna identifies these integrals with  graphs (see Figure~\ref{figure:Fig5}) --- graphs connecting the singularities of the function given by an integral against the exponential differentials.

 Now we consider the limit
$$ e^{p(z)/u}dz. $$
Geometrically limit differential can be identified with polynomials e.g. with Landau--Ginzburg models (see Figure~\ref{figure:Fig5}) --- for more see \cite{HKK}.

\begin{figure}[h]
\centering
\includegraphics[width=4cm]{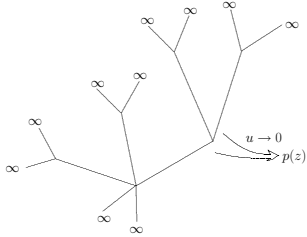}
\caption{Taking limit.}\label{figure:Fig5}
\end{figure}

\end{ex}

\subsection{The fiber over zero}

The fiber over zero (described in what follows)   plays an analogous role to the moduli space of Higgs Bundles in Simpson's  twistor family in the theory of nonabelian Hodge structures.
As it was alluded  earlier, an important class of examples of
categories and their stability conditions arises from Homological
Mirror Symmetry --- Fukaya--Seidel categories.  Indeed, such categories are the origin of the
modern definition of such stability conditions. The prescription given by Batyrev--Borisov and Hori--Vafa in \cite{BB}, \cite{hori-vafa}
to obtain homological mirrors for toric Fano varieties is
perfectly explicit and provides a reasonably large set of examples
to examine. We recall that if $\Sigma$ is a fan in $\RR^n$ for a
toric Fano variety $X_\Sigma$, then the homological mirror to the
B model of $X_\Sigma$ is a Landau--Ginzburg model $w\colon (\CC^*)^n \to \RR$ where
the Newton polytope $Q$ of $w$ is the fan polytope 
of $\Sigma$. In fact, we may consider the domain $(\CC^*)^n$ to
occur as the dense orbit of a toric variety $X_A$ where $A$ is $Q
\cap \ZZ^n$ and $X_A$ indicates the polytope toric construction. In
this setting, the function $w$ occurs as a pencil $V_w \subset H^0
(X_A, L_A)$ with fiber at infinity equal to the toric boundary of
$X_A$. Similar construction works for generic nontoric Fano's.
In this paper we work with  directed Fukaya category associated to the
superpotential $w$ --- Fukaya--Seidel categories.   To build on the
discussion above, we discuss here these two categories in the
context of stability conditions. The fiber over zero corresponds to the moduli of complex
structures. If $X_A$ is toric, the space of complex structures on
it is trivial, so the complex moduli appearing here are a result
of the choice of fiber $H \subset X_A$ and the choice of pencil
$w$ respectively. The appropriate stack parameterizing the choice
of fiber contains the quotient $[U/(\CC^*)^n]$ as an open dense
subset where $U$ is the open subset of $H^0 (X_A, L_A)$ consisting
of those sections whose hypersurfaces are nondegenerate (i.e.
smooth and transversely intersecting the toric boundary) and
$(\CC^*)^n$ acts by its action on $X_A$. To produce a reasonably
well behaved compactification of this stack, we borrow from the
work of Alexeev (see~\cite{Al02}), Gelfand--Kapranov--Zelevinsky (see~\cite{GKZ}), and Lafforgue (see~\cite{Lafforgue}) to
construct the stack $\mathcal{X}_{Sec (A)}$ with universal
hypersurface stack $\mathcal{X}_{Laf (A)}$. We quote the following
theorem which describes much of the qualitative behavior of these
stacks:

\begin{teo}[\cite{DKK}] i) The stack $\mathcal{X}_{Sec (A)}$ is a toric stack with moment polytope equal to the secondary polytope $Sec (A)$ of $A$. \\
ii) The stack $\mathcal{X}_{Laf (A)}$ is a toric stack with moment polytope equal to the Minkowski sum $Sec (A) + \Delta_A$ where $\Delta_A$ is the standard simplex in $\RR^A$. \\
ii) Given any toric degeneration $F: Y \to \CC$ of the pair $(X_A,
H)$, there exists a unique map $f : \CC \to \mathcal{X}_{Sec (A)}$
such that $F$ is the pullback of $\mathcal{X}_{Laf (A)}$.
\end{teo}

We note that in the theorem above, the stacks $\mathcal{X}_{Laf
(A)}$ and $\mathcal{X}_{Sec (A)}$ carry additional equivariant
line bundles that have not been examined extensively in existing
literature, but are of great geometric significance. The stack
$\mathcal{X}_{Sec (A)}$ is a moduli stack for toric degenerations
of toric hypersurfaces $H \subset X_A$. There is a hypersurface
$\mathcal{E}_A \subset \mathcal{X}_{Sec (A)}$ which parameterizes
all degenerate hypersurfaces. For the Fukaya category of
hypersurfaces in $X_A$, the compliment $\mathcal{X}_{Sec (A)}\setminus \mathcal{E}_A$ plays the role of the classical stability
conditions, while including $\mathcal{E}_A$ incorporates the
compactified version where MHS come into effect.

To find the stability conditions associated to the directed Fukaya
category of $(X_A, w)$, one needs to identify the complex
structures associated to this model. In fact, these are precisely
described as the coefficients of the superpotential, or in our
setup, the pencil $V_w \subset H^0 (X_A , w)$. Noticing that the
toric boundary is also a toric degeneration of the hypersurface,
we have that the pencil $V_w$ is nothing other than a map from
$\mathbb{P}^1$ to $\mathcal{X}_{Sec (A)}$ with prescribed point at
infinity. If we decorate $\mathbb{P}^1$ with markings at the
critical values of $w$ and $\infty$, then we can observe such a
map as an element of $\mathcal{M}_{0, Vol (Q) + 1}
(\mathcal{X}_{Sec (A)} , [w])$ which evaluates to $\mathcal{E}_A$
at all points except one and $\partial X_A$ at the remaining
point. We define the cycle of all stable maps with such an
evaluation to be $\mathcal{W}_A$ and regard it as the appropriate
compactification of complex structures on Landau--Ginzburg A models. Applying
techniques from fiber polytopes we obtain the following
description of $\mathcal{W}_A$:

\begin{teo}[\cite{DKK}] The stack $\mathcal{W}_A$ is a toric stack with moment polytope equal to the monotone path polytope of $Sec(A)$. \end{teo}

The polytope occurring here is not as widely known as the
secondary polytope, but occurs in a broad framework of so called
iterated fiber polytopes introduced by Billera and Sturmfels.

In addition to the applications of these moduli spaces to
stability conditions, we also obtain important information on the
directed Fukaya categories and their mirrors from this approach.
In particular, the above theorem may be applied to computationally
find a finite set of special Landau--Ginzburg models $\{w_1, \ldots, w_s\}$
corresponding to the fixed points of $\mathcal{W}_A$ (or the
vertices of the monotone path polytope of $Sec (A)$). Each such
point is a stable map to $\mathcal{X}_{Sec (A)}$ whose image in
moment space lies on the $1$-skeleton of the secondary polytope.
This gives a natural semiorthogonal decomposition of the directed
Fukaya category into pieces corresponding to the components in the
stable curve which is the domain of $w_i$. After ordering these
components, we see that the image of any one of them is a
multi-cover of the equivariant cycle corresponding to an edge of
$Sec (A)$. These edges are known as circuits in combinatorics and
we study the categories defined by each such component in
\cite{DKK}.

Now we put this moduli space as a ``zero fiber'' of the ``twistor'' family of
moduli family of stability conditions.

\begin{teo}[see \cite{KKP}] The fiber over zero
is a formal scheme $F$ over   $\mathcal{W}_A$ determined by the solutions of
the Mauer--Cartan equations for a dg-complex

$$
\begin{CD}
0 @<<< \Lambda^3 T_{\overline{Y}} @<<< \Lambda^2 T_{\overline{Y}} @<<< T_{\overline{Y}} @<<< \cO_{\overline{Y}} @<<< 0. \\[-3mm]
@. -3     @.             -2      @. -1       @. 0
\end{CD}
$$

\end{teo}

{\bf A sketch of the proof.} The above complex describes deformations with fixed fiber  at infinity.  We can associate with this complex a Batalin--Vilkovisky algebra. Following  \cite{KKP} we associate with it a smooth stack.
In the case of Fukaya--Seidel category of a    Landau--Ginzburg mirror of a Fano manifold $X$ the argument above implies that the dimension of the smooth stack of Landau--Ginzburg models is equal to $h^{1,1}(X) +1$.

We also have a $\CC^*$ action on $F$ with fixed points corresponding to limiting stability conditions.

\begin{conj}[see \cite{KKPS}] The local completion of fixed points over $X$ has a mixed Hodge structure.

\end{conj}

In the same way as the fixed point set under the  $\CC^*$ action play an important role in describing the rational homotopy types of smooth projective varieties we study the fixed points of the  $\CC^*$ action on $F$ and derive information about the homotopy type of a category. 

Similarly we can modify the above complex by fixing only a  part of the fiber
at infinity and deforming the rest. Similar  Batalin--Vilkovisky algebra technique allows us to prove

\begin{teo}(\cite{KKP})  We obtain a smooth moduli stack of Landau--Ginzburg models if we fix only a part of the fiber at infinity.
\end{teo}

This means that we can  allow different parts of the fiber at infinity to move
--- we call this part a \emph{moving scheme}. The geometrical properties  of the moving scheme contain  a deep birational, categorical, and  algebraic cycles information. We record this information in new invariants,
mainly emphasizing the birational content.

\subsection{Birational Applications}

In this subsection we look at the data collected from Sections~\ref{section:classical} and~\ref{subsubsection:MSVHS} from a new categorical prospective.

We apply the theorems above to the case of Landau--Ginzburg models for Del Pezzo surfaces --- this gives a new read of the Subsection~\ref{subsection:2dim basic links}. The basic links among Del Pezzo surfaces can be interpreted as follows.

\begin{teo}[\cite{DKK}]
\label{theorem:del Pezzo}
There exists an 11-dimensional moduli stack of all Landau--Ginzburg models of all del Pezzo surfaces. This moduli stack has a cell structure with the biggest cells corresponding to the   del Pezzo surfaces of big Picard rank. The basic links correspond to moving to the boundary of this stack.
\end{teo}

{\bf Proof.} The proof of this statement amounts to allowing all points at the fiber at infinity to move (the case of rational elliptic fibration)   and then fixing them one by one (for the first step see Table~\ref{tab:PICTURE}).

\begin{table}[h]
  \begin{center}
    \begin{tabular}[t]{|c|c|}
\hline
\begin{minipage}[c]{0.4\nanowidth}
\centering
\medskip

Del Pezzo.

\medskip

\end{minipage}
&
\begin{minipage}[c]{1.6\nanowidth}
\medskip

\begin{center}

Moving scheme.

\end{center}

\medskip

\end{minipage}
\\
\hline \hline
\begin{minipage}[c]{0.4\nanowidth}
\medskip

\begin{center}

$\PP^2$

\end{center}

\medskip

\end{minipage}
&
\begin{minipage}[c]{1.6\nanowidth}
\medskip

\begin{center}

$\includegraphics[width=5cm]{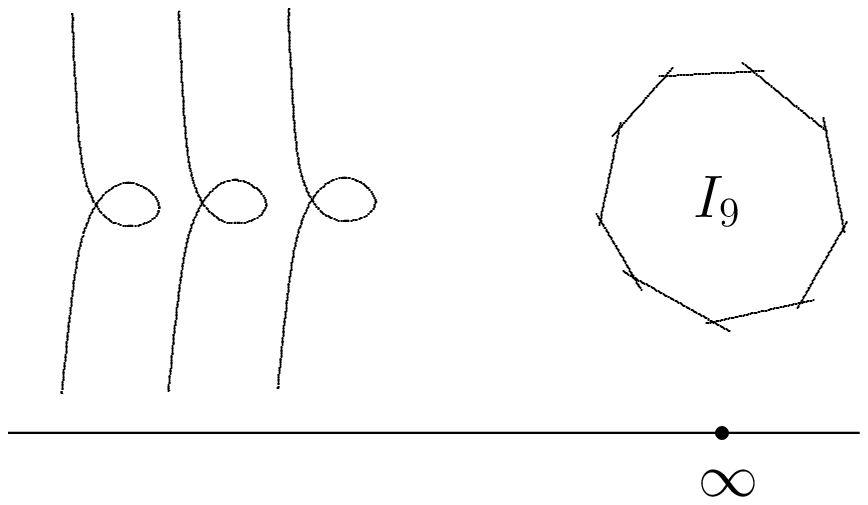}$

\end{center}

\medskip

\end{minipage}
\\ \hline
\begin{minipage}[c]{0.4\nanowidth}
\medskip

\begin{center}

$\widehat{P}^2_{pt}$

\end{center}

\medskip

\end{minipage}
&
\begin{minipage}[c]{1.6\nanowidth}
\medskip

\begin{center}

$\includegraphics[width=5cm]{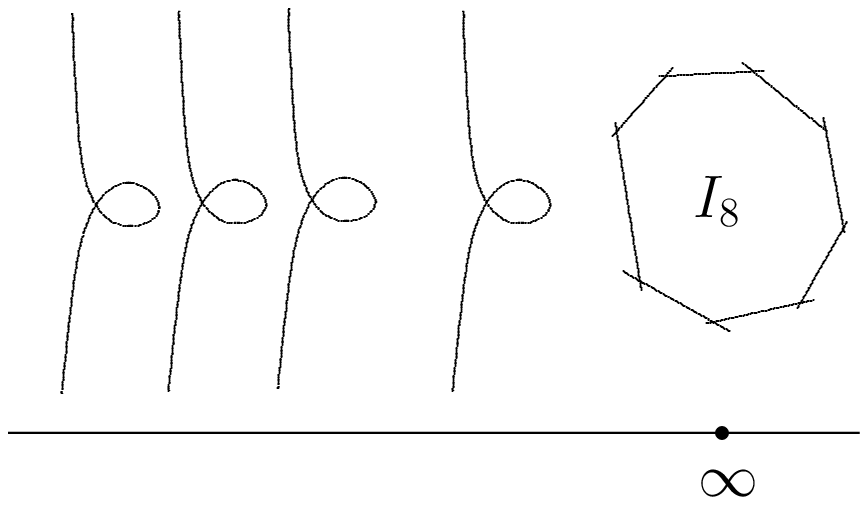}$

\end{center}

\medskip

\end{minipage}
\\ \hline
    \end{tabular}
    \caption{Moving points.}
    \label{tab:PICTURE}
  \end{center}
\end{table}

Theorem~\ref{theorem:del Pezzo} suggests that we can extend the  construction to rational blow-downs. We associate a moduli space of Landau--Ginzburg models to a rational blow-down of a rational surface by fixing corresponding  subschemes of the fiber at infinity.
This is a new construction in category theory, where the compactified moduli spaces of Landau--Ginzburg models plays the role of the moduli space of vector bundles in the Donaldson theory of polynomial invariants. As a result we get a tool for studying the semiorthogonal decompositions by putting a topological structure on them based on the compactification of moduli spaces of Landau--Ginzburg models.
We conjecture the following (see also \cite{DKK2}).

\begin{conj}[see \cite{KKPS}] The derived categories of the Barlow surface  and of the rational blow-down described above contain as a semiorthogonal piece a  phantom category
i.e. a nontrivial category with trivial $K^0$ group.

\end{conj}

This conjecture is rather bold and will make the studies of algebraic cycles and rationality questions rather difficult. Some evidence for it have already appeared in \cite{IKP}.

We summarize our findings in Table~\ref{tab:LGF1}.

\begin{table}[h]
  \begin{center}
    \begin{tabular}[t]{|c|c|c|}
\hline
\begin{minipage}[c]{0.8\nanowidth}
\centering
\medskip

Surface

\medskip

\end{minipage}
&
\begin{minipage}[c]{0.8\nanowidth}
\medskip

\begin{center}

Configuration at $\infty$ and a moving scheme

\end{center}

\medskip

\end{minipage}
&
\begin{minipage}[c]{0.8\nanowidth}
\medskip

\begin{center}

$M_{LG}$ as an invariant

\end{center}

\medskip

\end{minipage}
\\
\hline \hline
\begin{minipage}[c]{0.8\nanowidth}
\medskip

\begin{center}

$\PP^2$.

\end{center}

\medskip

\end{minipage}
&
\begin{minipage}[c]{0.8\nanowidth}
\medskip

\begin{center}

Wheel of 9 lines, all points are fixed --- no moving scheme.

\end{center}

\medskip

\end{minipage}
&
\begin{minipage}[c]{0.8\nanowidth}
\medskip

\begin{center}

Two-dimensional moduli space, $FS(LG(\PP^2))$ has a semiorthogonal decomposition.

\end{center}

\medskip

\end{minipage}
\\ \hline
\begin{minipage}[c]{0.8\nanowidth}
\medskip

\begin{center}

$E(1)$ --- rational elliptic surface with 12 singular fibers.

\end{center}

\medskip

\end{minipage}
&
\begin{minipage}[c]{0.8\nanowidth}
\medskip

\begin{center}

Wheel of 9 lines, all points can move --- moving scheme is all intersection points.

\end{center}

\medskip

\end{minipage}
&
\begin{minipage}[c]{0.8\nanowidth}
\medskip

\begin{center}

Ten dimensional moduli space.

\end{center}

\medskip

\end{minipage}
\\ \hline
\begin{minipage}[c]{0.8\nanowidth}
\medskip

\begin{center}

Rational blow-down of $E(1)\#7\overline{\CC\PP}^2$.

\end{center}

\medskip

\end{minipage}
&
\begin{minipage}[c]{0.8\nanowidth}
\medskip

\begin{center}

Wheel of 9 lines s.t. six points with multiplicities configure a moving scheme. 

\end{center}

\medskip

\end{minipage}
&
\begin{minipage}[c]{0.8\nanowidth}
\medskip

\begin{center}

The moduli space is parameterized by Exts't from $\Phi$ to $E_i$ in the category
$\langle E_1, \ldots, E_9, \Phi\rangle$,

$K^0(\Phi)=0$,

$\pi_1 (M_{LG(E(1)\#7\overline{\CC\PP}^2)}\setminus D)$ as an invariant.

\end{center}

\medskip

\end{minipage}
\\ \hline
%
%
%
%
%
%
%
%
%
%
%
%
%
%
%
%
%
%
%
%
%
%
%
%
%
%
%
%
%
%
\begin{minipage}[c]{0.8\nanowidth}
\medskip

\begin{center}

$\CC\PP^2 \# 8 \overline{\CC\PP}^2$.

\end{center}

\medskip

\end{minipage}
&
\begin{minipage}[c]{0.8\nanowidth}
\medskip

\begin{center}

Wheel of 9 curves, the moving scheme consist of 8 points.

\end{center}

\medskip

\end{minipage}
&
\begin{minipage}[c]{0.8\nanowidth}
\medskip

\begin{center}

The moduli space is parameterized by Exts't among $E_i$'s in the category
$\langle E_1, \ldots, E_{11}\rangle$,
$\pi_1 (M_{LG(\CC\PP^2 \# 8 \overline{\CC\PP}^2)}\setminus D)$ as an invariant.

\end{center}

\medskip

\end{minipage}
\\ \hline
    \end{tabular}
    \caption{Moduli of Landau--Ginzburg models for surfaces.}
    \label{tab:LGF1}
  \end{center}
\end{table}

In the leftmost part of this table we consider different surfaces. In second part we describe the fiber at infinity with the corresponding moving scheme.
In the last part we comment what is  the moduli of Landau--Ginzburg models
and what are some of its invariants. In most cases this is the fundamental group of the non-compactified moduli space. In case of rational blow-down this fundamental group suggests the appearance of new phenomenon a nontrivial category with trivial $K^0$ group --- a phantom category, which we will discuss later. This also appear in the Barlow surface. The connection with Godeaux surface (see \cite{BBS}) suggests that   the fundamental group of the non-compactified moduli space differs form the  the fundamental group of the non-compactified moduli space
of LG models for Del Pezzo surface  of degree 1.

\begin{oss} Figure~\ref{figure:Fano snake} suggests that different Fanos are connected in the big  moduli of Landau--Ginzburg models either by wall-crossings or by going to the boundary of such a moduli space.
\end{oss}

\subsection{High  dimensional Fano's}

We concentrate on the case of high dimensional Fano manifolds. We give the findings in Subsection~\ref{subsection:3dim basic links} in the following categorical read.

\begin{conj} There exists a moduli stack of  Landau--Ginzburg models of all 3-dimensional Fano's. It has  a cell structure parallel to the basic links from Subsection~\ref{subsection:3dim basic links}.
\end{conj}

This conjecture is based on the following implementation of the theory of Landau--Ginzburg models. In the same way as in case of del Pezzo surfaces we can allow moving different subschemes at the fiber at infinity. The first 3-dimensional examples was worked out at \cite{KKOY} and \cite{AAK}. In these cases the moving scheme at infinity   corresponds to a Riemann surface  so modified Landau--Ginzburg models correspond to Landau--Ginzburg mirrors of blown up toric varieties.
In higher dimensions of course the cell structure is more elaborated.
By fixing different  parts at the divisor at infinity we can change the Picard rank of the generic fiber. Modifications, gluing, and conifold transitions are needed in dimension three and four. These leads to  the need of Landau--Ginzburg moduli spaces with many components.

The next case to consider is the case of three-dimensional cubic.
In this case the moving scheme is described at \cite{IKP}.
Similar moving scheme is associated with the threefold  $X_{14}$.

The next theorem follows from~\cite{DKK}.

\begin{teo} The moduli space of the Landau--Ginzburg mirrors for the smooth
three dimensional cubic and  $X_{14}$ are deformations of one another.
\end{teo}

As it follows from \cite{DKK} this would imply their birationality since it means that some  Mori fibrations associated with three-dimensional cubic and  $X_{14}$ are connected via a Sarkisov links.
The A side interpretation of this result is given in \cite{BFK2}. It implies that the semiorthogonal decompositions of the derived categories of three-dimensional cubic and  $X_{14}$ have  a common semiorthogonal piece and differ only by several exceptional objects --- a result obtained by  Kuznetsov in \cite{KUZ}.
Similar observations can be made for other three-dimensional Fano manifolds, whose Landau--Ginzburg models can be included in one big moduli space. So studying and comparing  these Landau--Ginzburg models at the same time brings a new approach to birational geometry.  The material described in Subsection~\ref{subsection:3dim basic links} suggests that there are many other 3-dimensional Fano manifolds related as  3-dimensional cubic and $X_{14}$, that is related
by one only non-commutative cobordism.
Moving from one Landau--Ginzburg model associated to one Fano threefold to another can be
considered as a  certain ``wall-crossing''. As the material of Subsection~\ref{subsection:3dim basic links} suggests we can include singular Fano threefolds as boundary of the moduli space of
Landau--Ginzburg models --- i.e. ``limiting stability conditions'' on which even more dramatic ``wall-crossing'' occurs.
The experimental material from  Subsection~\ref{subsection:3dim basic links}  and Conjecture~\ref{conjecture:MS} also brings the idea that  studying birational geometry of Fano threefolds and proving Homological Mirror Symmetry for them might be closely related problems.

Similar picture exists in higher dimension. We give  examples and invariants connected with  moduli spaces of Landau--Ginzburg models associated with very special  4-dimensional Fano manifolds --- four dimensional cubic and their ``relatives''. For these 4-folds there we look at  so called Hasset--Kuznetsov--Tschinkel  program from Landau--Ginzburg prospective.

It is expected that there are many analogues in dimension four to the behavior of three-dimensional cubic and  $X_{14}$, namely they have a common semiorthogonal piece and differ only by several exceptional objects. We indicate several of them.

The study of four dimensional cubic was undertaken by many people: Voisin, Beauville, Donagi, Hasset, Tschinkel. On derived  category level  a lot of fundamental work was done by Kuznetsov and then  Addington and  Thomas. On the Landau--Ginzburg side calculations were done by \cite{KP} and \cite{IKS}.  We extend the approach we have undertaken in the case of three-dimensional Fano threefolds to the case of some fourfolds.

We recall the following theorem by Kuznetsov.

\begin{teo}[\cite{KUZ}] Let $X$ be a smooth 4-dimensional cubic.
Then

$$D^b(X)=\langle D^b(K3), E_1,  E_2,  E_3\rangle.$$

\end{teo}

Here $D^b(K3)$ is the  derived category of a noncommutative $K3$  surface.
This noncommutative $K3$ surface is very non-generic one.

Moving  subscheme at infinity  corresponding to a generic $K3$ surface we obtain the moduli spaces of Landau--Ginzburg models associated to
four-dimensional $X_{10}$.
This suggests

\begin{conj} Let $X$ be a smooth 4-dimensional variety $X_{10}$.
Then
$$D^b(X)=\langle D^b(K3), E_1,  E_2,  E_3, E_4\rangle.$$
\end{conj}

Here $D^b(K3)$ is the  derived category of a generic noncommutative $K3$  surface. We expect this conjecture will follow from some version of homological projective duality.

Similarly  to the 3-dimensional case there is overlap between the Landau--Ginzburg moduli spaces of the four-dimensional cubic  and four-dimensional $X_{10}$.  We conjecture

\begin{conj} There is an infinite series of moduli of Landau--Ginzburg models associated with special (from the Noether--Lefschetz loci) four-dimensional cubics  and four-dimensional $X_{10}$'s, which can be deformed one to another and therefore they are birational (see Table~\ref{tab:SSARK}).
\end{conj}

\begin{table}[h]
  \begin{center}
    \begin{tabular}[t]{|c|c|}
\hline
\begin{minipage}[c]{0.2\nanowidth}
\centering
\medskip

Dim.

\medskip

\end{minipage}
&
\begin{minipage}[c]{2\nanowidth}
\medskip

\begin{center}

Landau--Ginzburg moduli behavior

\end{center}

\medskip

\end{minipage}
\\
\hline \hline
\begin{minipage}[c]{0.2\nanowidth}
\medskip

\begin{center}

3

\end{center}

\medskip

\end{minipage}
&
\begin{minipage}[c]{2\nanowidth}
\medskip

\begin{center}

$\includegraphics[width=7cm]{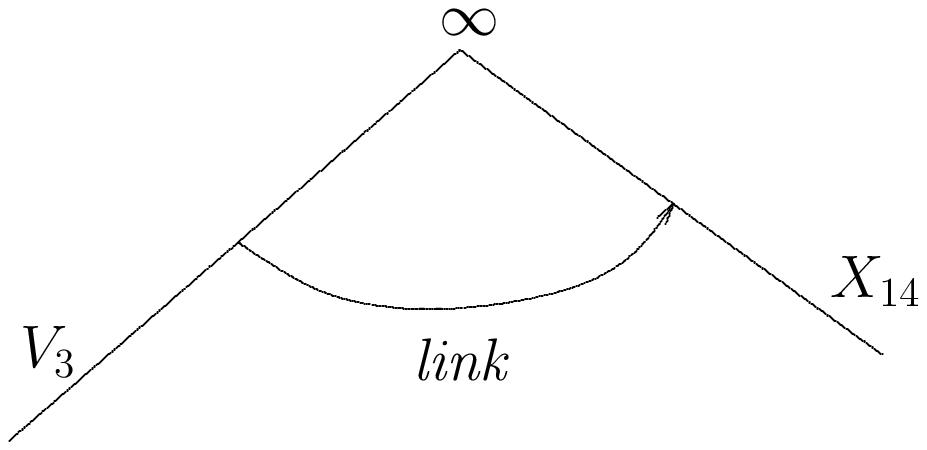}$

\end{center}

\medskip

\end{minipage}
\\ \hline
\begin{minipage}[c]{0.2\nanowidth}
\medskip

\begin{center}

4

\end{center}

\medskip

\end{minipage}
&
\begin{minipage}[c]{2\nanowidth}
\medskip

\begin{center}

$\includegraphics[width=7cm]{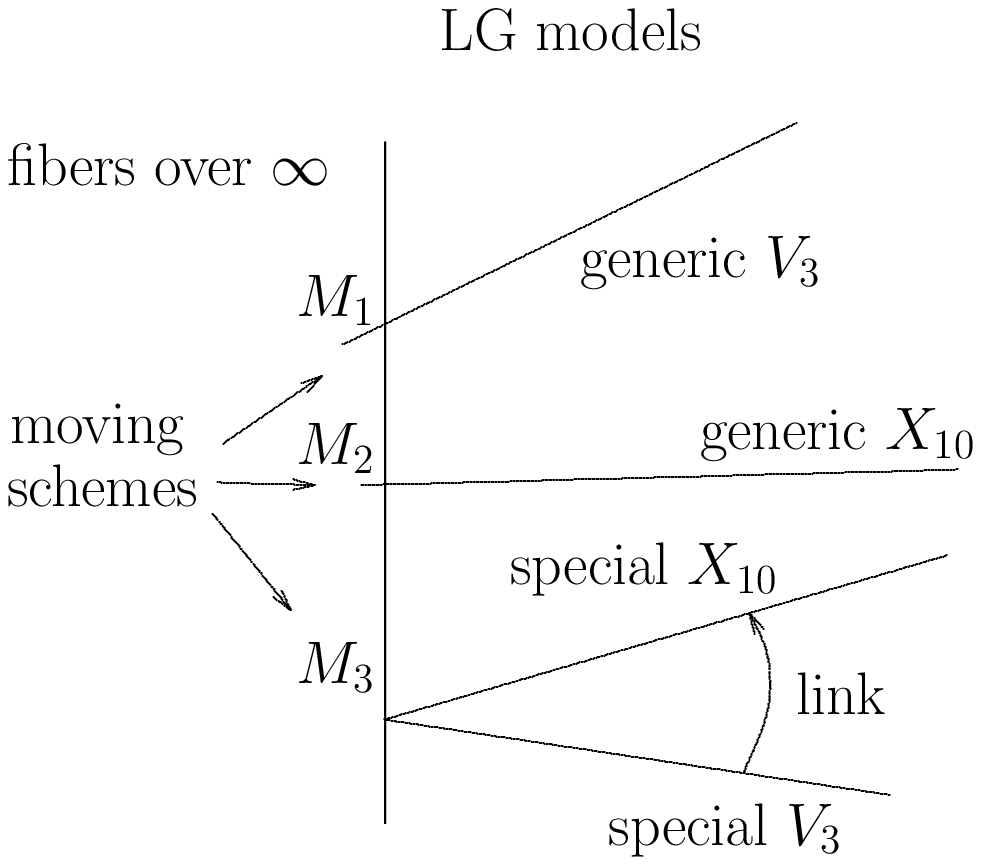}$

\end{center}

\medskip

\end{minipage}
\\ \hline
    \end{tabular}
    \caption{Non-commutative Sarkisov program.}
    \label{tab:SSARK}
  \end{center}
\end{table}

This series corresponds to cases when the moving scheme at infinity is associated with commutative $K3$ surface. According to \cite{DKK} such  a deformation between moduli of Landau--Ginzburg models implies  the existence of a Sarkisov links between such cubics and $X_{10}$. We will return to  rationality questions  in the next subsection. We would like to mention here that a generalization of  homological projective duality of Kuznetsov's  arrives at

\begin{conj}
Let $X$ be a smooth 4-dimensional  Kuechle manifold (see~\cite{KUECH})
Then
$$D^b(X)=\langle D^b(K3), E_1,  E_2, \ldots, E_n\rangle.$$
\end{conj}

Here $D^b(K3)$ is the  derived category of a generic noncommutative $K3$  surface. As a consequence we have

\begin{conj} There is an infinite series of moduli of Landau--Ginzburg models associated with special  four-dimensional cubics, four-dimensional $X_{10}$'s, and  Kuechle manifolds, which can be deformed one to another.
\end{conj}

This conjecture suggests that rationality question for Kuechle manifolds can be treated similarly as the questions for  four dimensional cubic  and $X_{10}$ --- see the next subsection.

\subsection{Invariants}

In this subsection we introduce two types of invariants, which are connected with SHS and moduli spaces of
Landau--Ginzburg models.

The first type is a global invariant --- Orlov spectra of a category.
It was conjectured in \cite{BFK} that it is an invariant measuring nonrationality.   Its relation to Landau--Ginzburg models was emphasized in \cite{KP},
\cite{IKP}, \cite{IKS}. In this subsection we relate it to the Hasset---Kuznetsov---Tschinkel  program --- a program relating the Noether--Lefschetz components to rationality of 4-dimensional cubic.

The second type of invariant is of local nature --- the local singularity of the Landau--Ginzburg models. We relate these invariants to stability conditions.
We suggest that they play the role of discrepancies and thresholds in the Kawamata's correspondence described in Table~\ref{tab:CPINTRO}. In other words these invariants measure if two Landau--Ginzburg moduli spaces can be deformed one to another and according to
\cite{DKK} if there are  Sarkisov links connecting the Fano manifolds from the A side.

The noncommutative Hodge structures were introduced by  Kontsevich, Katzarkov, and Pantev in \cite{KKP}  as  means of bringing
the techniques and tools of Hodge theory into the categorical and
noncommutative realm.  In the classical setting, much of the
information about an isolated singularity is recorded by means of
the Hodge spectrum, a set of rational eigenvalues of the monodromy
operator.  The Orlov spectrum (defined below), is a categorical
analogue of this Hodge spectrum appearing in the works of Orlov (\cite{O}) and
Rouquier (\cite{R}). The missing numbers in
the spectra are called gaps.

Let $\mathcal T$ be a triangulated category.  For any $G \in
\mathcal T$ denote by $\langle G \rangle_0$ the smallest full
subcategory containing $G$ which is closed under isomorphisms,
shifting, and taking finite direct sums and summands. Now
inductively define $\langle G \rangle_n$ as the full subcategory
of objects, $B$, such that there is a distinguished triangle, $X
\to B \to Y \to X[1]$, with $X \in \langle G \rangle_{n-1}$ and $Y
\in \langle G \rangle_0$.

\begin{definition}
Let $G$ be an object of a triangulated category $\mathcal{T}$.  If
there is some number $n$ with $\langle G \rangle_{n} = \mathcal T$, we set
\begin{displaymath}
 t(G):= \text{min } \lbrace n \geq 0 \ | \ \langle G
 \rangle_{n} = \mathcal T \rbrace.
\end{displaymath}
Otherwise, we set $t(G) := \infty$.  We call $t(G)$
the \emph{generation time} of $G$. If $t(G)$ is finite,
we say that $G$ is a \emph{strong generator}. The \emph{Orlov
spectrum} of $\mathcal T$ is the union of all possible generation
times for strong generators of $\mathcal T$.  The \emph{Rouquier
dimension} is the smallest number in the Orlov spectrum.  We say
that a triangulated category, $\mathcal T$, has a \emph{gap} of
length $s$ if $a$ and $a+s+1$ are in the Orlov spectrum but $r$
is not in the Orlov spectrum for $a < r < a+s+1$.
\end{definition}

The first connection to Hodge theory appears in the form of the
following theorem.
\begin{teo}[see \cite{BFK}]
Let $X$ be an algebraic variety possessing an isolated
hypersurface singularity. The Orlov spectrum of the category of
singularities of $X$ is bounded by twice the embedding dimension
times the Tjurina number of the singularity.
\label{thm:isohypspecbound}
\end{teo}

We also recall the following conjecture which will play an important role in our considerations.

\begin{conj}[see \cite{BFK}] Let $X$ be a rational Fano manifold of  dimension $n>2$.   Then a gap of spectra of $D^b(X)$ is less or equal to $n-3$.
\end{conj}

After this brief review of theory of spectra and gaps we connect them with
the SHS and moduli of Landau--Ginzburg models.

\begin{conj}[see \cite{KKP}] The monodromy of the Landau--Ginzburg models for Fano manifold $X$ determines the gap of spectra of $D^b(X)$.

\end{conj}

This conjecture was partially verified in \cite{KP}, \cite{IKP}, \cite{IKS}.

We record our findings in Table~\ref{tab:LGF3}.

\begin{table}[h]
  \begin{center}
    \begin{tabular}[t]{|c|c|c|c|c|}
\hline
\begin{minipage}[c]{0.2\nanowidth}
\centering
\medskip

Dim.

\medskip

\end{minipage}
&
\begin{minipage}[c]{0.6\nanowidth}
\medskip

\begin{center}

Examples

\end{center}

\medskip

\end{minipage}
&
\begin{minipage}[c]{0.3\nanowidth}
\centering
\medskip

Hodge diamond

\medskip

\end{minipage}
&

\begin{minipage}[c]{\nanowidth}
\medskip

\begin{center}

Categories

\end{center}

\medskip

\end{minipage}
&
\begin{minipage}[c]{0.5\nanowidth}
\medskip

\begin{center}

Invariants

\end{center}

\medskip

\end{minipage}
\\ \hline\hline
\begin{minipage}[c]{0.2\nanowidth}
\centering
\medskip

2

\medskip

\end{minipage}
&
\begin{minipage}[c]{0.6\nanowidth}
\medskip

\begin{center}

Rational blow-down $X_1$ of
$\CC\PP^2\#6\overline{\CC\PP}^2$

\end{center}

\medskip

\end{minipage}
&
\begin{minipage}[c]{0.3\nanowidth}
\centering
\medskip

$
\begin{array}{c}
  1 \\
  7 \\
  1
\end{array}
$

\medskip

\end{minipage}
&
\begin{minipage}[c]{\nanowidth}
\medskip

\begin{center}

$D^b(X_1)=\langle E_1,\ldots, E_9,\mathcal A\rangle$, $K^0(\mathcal A)=0$

\end{center}

\medskip

\end{minipage}
&
\begin{minipage}[c]{0.5\nanowidth}
\medskip

\begin{center}

$\pi_1(M_{LG}\setminus D)$

\end{center}

\medskip

\end{minipage}
\\ \hline
\begin{minipage}[c]{0.2\nanowidth}
\centering
\medskip

2

\medskip

\end{minipage}
&
\begin{minipage}[c]{0.6\nanowidth}
\medskip

\begin{center}

Barlow surface
$X_2=\CC\PP^2\#8\overline{\CC\PP}^2$

\end{center}

\medskip

\end{minipage}
&
\begin{minipage}[c]{0.3\nanowidth}
\centering
\medskip

$
\begin{array}{c}
  1 \\
  9 \\
  1
\end{array}
$

\medskip

\end{minipage}
&
\begin{minipage}[c]{\nanowidth}
\medskip

\begin{center}

$D^b(X_2)=\langle E_1,\ldots, E_{11},\mathcal A\rangle$, $K^0(\mathcal A)=0$

\end{center}

\medskip

\end{minipage}
&
\begin{minipage}[c]{0.5\nanowidth}
\medskip

\begin{center}

$\pi_1(M_{LG}\setminus D)$

\end{center}

\medskip

\end{minipage}
\\ \hline
\begin{minipage}[c]{0.2\nanowidth}
\centering
\medskip

3

\medskip

\end{minipage}
&
\begin{minipage}[c]{0.6\nanowidth}
\medskip

\begin{center}

Cubic threefold $X_1$

\end{center}

\medskip

\end{minipage}
&
\begin{minipage}[c]{0.3\nanowidth}
\centering
\medskip

$
\begin{array}{ccc}
   & 1 &   \\
   & 1 &   \\
 5 &   & 5 \\
   & 1 &   \\
   & 1 &
\end{array}
$
\medskip

\end{minipage}
&
\begin{minipage}[c]{\nanowidth}
\medskip

\begin{center}

$D^b(X_1)=\langle E_1, E_2, \mathcal A\rangle$

\end{center}

\medskip

\end{minipage}
&
\begin{minipage}[c]{0.5\nanowidth}
\medskip

\begin{center}

Monodromy  of Landau--Ginzburg models, gap in spectra is at most 1.

\end{center}

\medskip

\end{minipage}
\\ \hline
\begin{minipage}[c]{0.2\nanowidth}
\centering
\medskip

3

\medskip

\end{minipage}
&
\begin{minipage}[c]{0.6\nanowidth}
\medskip

\begin{center}

Artin--Mumford example $X_2$

\end{center}

\medskip

\end{minipage}
&
\begin{minipage}[c]{0.3\nanowidth}
\centering
\medskip

$
\begin{array}{c}
    1    \\
    1    \\
    1    \\
    1
\end{array}
$
\medskip

\end{minipage}
&
\begin{minipage}[c]{\nanowidth}
\medskip

\begin{center}

$D^b(X_2)=\langle E_1,\ldots, E_{10}, \mathcal A\rangle$, $K^0(\mathcal A)=\ZZ_2$

\end{center}

\medskip

\end{minipage}
&
\begin{minipage}[c]{0.5\nanowidth}
\medskip

\begin{center}

Monodromy  of Landau--Ginzburg models, gap in spectra is 0.

\end{center}

\medskip

\end{minipage}
\\ \hline
\begin{minipage}[c]{0.2\nanowidth}
\centering
\medskip

4

\medskip

\end{minipage}
&
\begin{minipage}[c]{0.6\nanowidth}
\medskip

\begin{center}

Cubic fourfold $X$

\end{center}

\medskip

\end{minipage}
&
\begin{minipage}[c]{0.3\nanowidth}
\centering
\medskip

$$
\begin{array}{ccc}
    &  1  &   \\
    &  1  &   \\
  1 & 21  & 1 \\
    &  1  &   \\
    &  1  &
\end{array}
$$

\medskip

\end{minipage}
&
\begin{minipage}[c]{\nanowidth}
\medskip

\begin{center}

$D^b(X)=\langle E_1, E_2, E_3, \mathcal A \rangle$

\end{center}

\medskip

\end{minipage}
&
\begin{minipage}[c]{0.5\nanowidth}
\medskip

\begin{center}

Monodromy of Landau--Ginzburg models, gap in spectra is at most 2.

\end{center}

\medskip

\end{minipage}
\\ \hline
    \end{tabular}
    \caption{Summary.}
    \label{tab:LGF3}
  \end{center}
\end{table}

As it is clear from our construction the monodromy of Landau--Ginzburg models depend on the choice of moving scheme. This suggests that the classical Hodge theory cannot distinguish rationality. We employ the geometry of the moduli spaces of Landau--Ginzburg models in order to do so. These moduli spaces measure the way the pieces in the semiorthogonal decompositions are put together --- this information computes the spectra of a category.

This was first observed in
 \cite{KP}, \cite{IKP}, and \cite{IKS}. Applying the theory of Orlov's spectra to the case of
four-dimensional cubic, four-dimensional $X_{10}$, and Kuechle manifolds we arrive at the following conjecture suggested by Hasset--Kuznetsov--Tschinkel program.

\begin{conj}
\label{conjecture:non-rational fourfolds}
The four-dimensional cubic, four-dimensional $X_{10}$, and Kuechle manifolds are not rational if they do not contain derived categories of commutative K3 surfaces in their semiorthogonal decompositions (see~\cite{KUZ} for cubic fourfold).
\end{conj}

In other words this conjecture implies that the generic of described above  fourfolds is not rational since the gap of their categories of coherent sheaves is equal to two.

In  case the semiorthogonal decompositions contain derived category of commutative K3 surface --- the issue is more delicate and requires the use of a Noether--Lefschetz spectra ---  see \cite{IKP}.

\begin{oss}
We are very grateful to A. Iliev who  has informed us that some checks of Conjecture~\ref{conjecture:non-rational fourfolds} were done by him, Debarre, and Manivel.
\end{oss}

We proceed with a topic which we have started in \cite{IKP} --- how to detect  rationality  when gaps of spectra cannot be used. The example we have considered there was the Artin--Mumford example. Initially it was shown that
the Artin--Mumford example is not rational since  it has a torsion in its third cohomology group. Our conjectural interpretation in \cite{IKP}
is that the Artin--Mumford example is not rational since  it is a conic bundle which contains the derived category of an Enriques surface in the SOD of its derived category. The derived category of an Enriques surface has 10 exceptional objects and a category  $\mathcal A$,  which does not look as a category of a curve, in its SOD. We conjecture in
 \cite{IKP} that derived category of the Artin--Mumford example has no gap in it spectra but it is the moving scheme which determines its nonrationality. We also exhibit the connection between  the category  $\mathcal A$ and its moving scheme --- see also \cite{IKUZ}.

We bring a totally new prospective to rationality questions  ---  the parallel  of spectra and gaps of categories with topological superconductors. Indeed if we consider  generators as Hamiltonians and generation times as states of matter we get a far-reaching parallel.
The first application of this parallel was a prediction of existence of phantom categories which we have defined above.
The phantoms are the equivalent of topological superconductors in the above parallel, which,  we conjecture, allows us to compute spectra in the same way as Turaev--Viro procedure allow us to compute  topological states in the Kitaev--Kong models. In fact the parallel produces a new spectra code which can be used in quantum computing opening new horizons for research.  We outline this parallel in Table~\ref{tab:WITT}.

\begin{table}[h]
  \begin{center}
    \begin{tabular}[t]{|c|c|}
\hline
\begin{minipage}[c]{1.2\nanowidth}
\centering
\medskip

Topological states of matter.

\medskip

\end{minipage}
&
\begin{minipage}[c]{1.2\nanowidth}

\medskip

\centering


Gaps and phantoms.


\medskip

\end{minipage}
\\
\hline \hline
\begin{minipage}[c]{1.2\nanowidth}

\centering

\medskip


Hamiltonians.


\medskip

\end{minipage}
&
\begin{minipage}[c]{1.2\nanowidth}

\centering

\medskip


Generators of the category.


\medskip

\end{minipage}
\\ \hline
\begin{minipage}[c]{1.2\nanowidth}

\centering

\medskip


Topological states.


\medskip

\end{minipage}
&
\begin{minipage}[c]{1.2\nanowidth}

\centering

\medskip


Generation time.


\medskip

\end{minipage}
\\ \hline
\begin{minipage}[c]{1.2\nanowidth}

\centering

\medskip


3-manifolds
in Kitaev--Kong model.


\medskip

\end{minipage}
&
\begin{minipage}[c]{1.2\nanowidth}

\centering

\medskip


Singularities of Landau--Ginzburg model.


\medskip

\end{minipage}
\\ \hline
\begin{minipage}[c]{1.2\nanowidth}

\centering

\medskip


Topological superconductors.


\medskip

\end{minipage}
&
\begin{minipage}[c]{1.2\nanowidth}

\centering

\medskip


Phantoms as limits of gaps.


\medskip

\end{minipage}
\\ \hline
    \end{tabular}
    \caption{Gaps, spectra and topological superconductors.}
    \label{tab:WITT}
  \end{center}
\end{table}

The existence of phantoms was  against the expectations of the founding fathers of derived categories.   Today it is known  that the phantoms are everywhere
in the same way as the topological insulators and topological superconductors --- a truly  ground --- breaking  unconventional parallel.
We anticipate striking applications of phantoms in the study of rationality of algebraic varieties. We briefly outline one of these applications.
We consider another  conic bundle ---  Sarkisov's example, see
\cite{SAR} and Table~\ref{tab:CGBT}. This example can be described
as follows --- we start with an irreducible singular plane curve
$C_{sing}$  in $\PP^2$ of degree $d\geqslant 3$  that has exactly
$(d-1)(d-2)/2-1$ ordinary double points (such curves exists for
every $d\geqslant 3$). Then we blow up $\PP^2$ at the singular
points of $C_{sing}$. Denote by $S$ the obtained surface and by
$C$ the proper transform of the curve $C_{sing}$. Then $C$ is a
smooth elliptic curve (easy genus count). Let
$\tau\colon\widetilde{C}\to C$ be some unramified double cover. Then
it follows from \cite[Theorem~5.9]{SAR} there exists a smooth
threefold $X$ of Picard rank $\mathrm{rk}(\mathrm{Pic}(S))+1$ with
a morphism $\pi\colon X\to S$ whose general fiber is
$\mathbb{P}^1$, i.e. $\pi$ is a conic bundle, such that $C$ is the
discriminant curve of $\pi$, and $\tau$ is induced by
interchanging components of the fibers over the points of $C$.
Moreover, it follows from \cite[Theorem~4.1]{SAR} that $X$ is not
rational if $d\geqslant 12$. On the other hand, we always have
$H^{3}(X,\mathbb{Z})=0$, since $C$ is an elliptic curve. Note that
birationally $X$ can be obtained as a degeneration of a standard
conic bundle over $\mathbb{P}^2$ whose discriminant curve is a
smooth curve of degree $d$. On the side of Landau--Ginzburg models
we can observe the following. The mirrors of conic bundles are
partially understood --- see \cite{AAK}. The degeneration
procedure on the B side amounts to conifold transitions on the A
side. These conifold transitions define    a moving scheme for the
Landau--Ginzburg model which suggests the following conjecture.

\begin{conj} The SOD of $D^b(X)$ 
contains a phantom category --- a nontrivial category with trivial $K^0$ group.

\end{conj}

This  phantom category is the reason for nonrationality of Sarkisov's  conic bundle, which we  conjecture  has no gaps in the spectra of its derived category --- see Table~\ref{tab:CGBT}.

\begin{table}[h]
  \begin{center}
    \begin{tabular}[t]{|c|c|c|}
\hline
\begin{minipage}[c]{0.6\nanowidth}
\centering
\medskip

Conic bundle.

\medskip

\end{minipage}
&
\begin{minipage}[c]{0.8\nanowidth}
\centering
\medskip


Hodge diamond.


\medskip

\end{minipage}
&
\begin{minipage}[c]{1.4\nanowidth}
\centering
\medskip

Phantom category.

\medskip

\end{minipage}
\\
\hline \hline
\begin{minipage}[c]{0.6\nanowidth}
\medskip

\begin{center}


Conic bundle $X$ 
with degeneration curve $C$,
$p_a(C)=1$.

\end{center}

\medskip

\end{minipage}
&
\begin{minipage}[c]{0.8\nanowidth}
\medskip

\begin{center}

$
\begin{array}{ccccccc}
  &   &   & 1  &   &   &   \\
  &   & 0 &    & 0 &   &   \\
  & 0 &   & 55 &   & 0 &   \\
0 &   & 0 &    & 0 &   & 0 \\
  & 0 &   & 55 &   & 0 &   \\
  &   & 0 &    & 0 &   &   \\
  &   &   & 1  &   &   &
\end{array}
$

\end{center}

\medskip

\end{minipage}
&
\begin{minipage}[c]{1.4\nanowidth}
\medskip

\begin{center}

$\includegraphics[width=5cm]{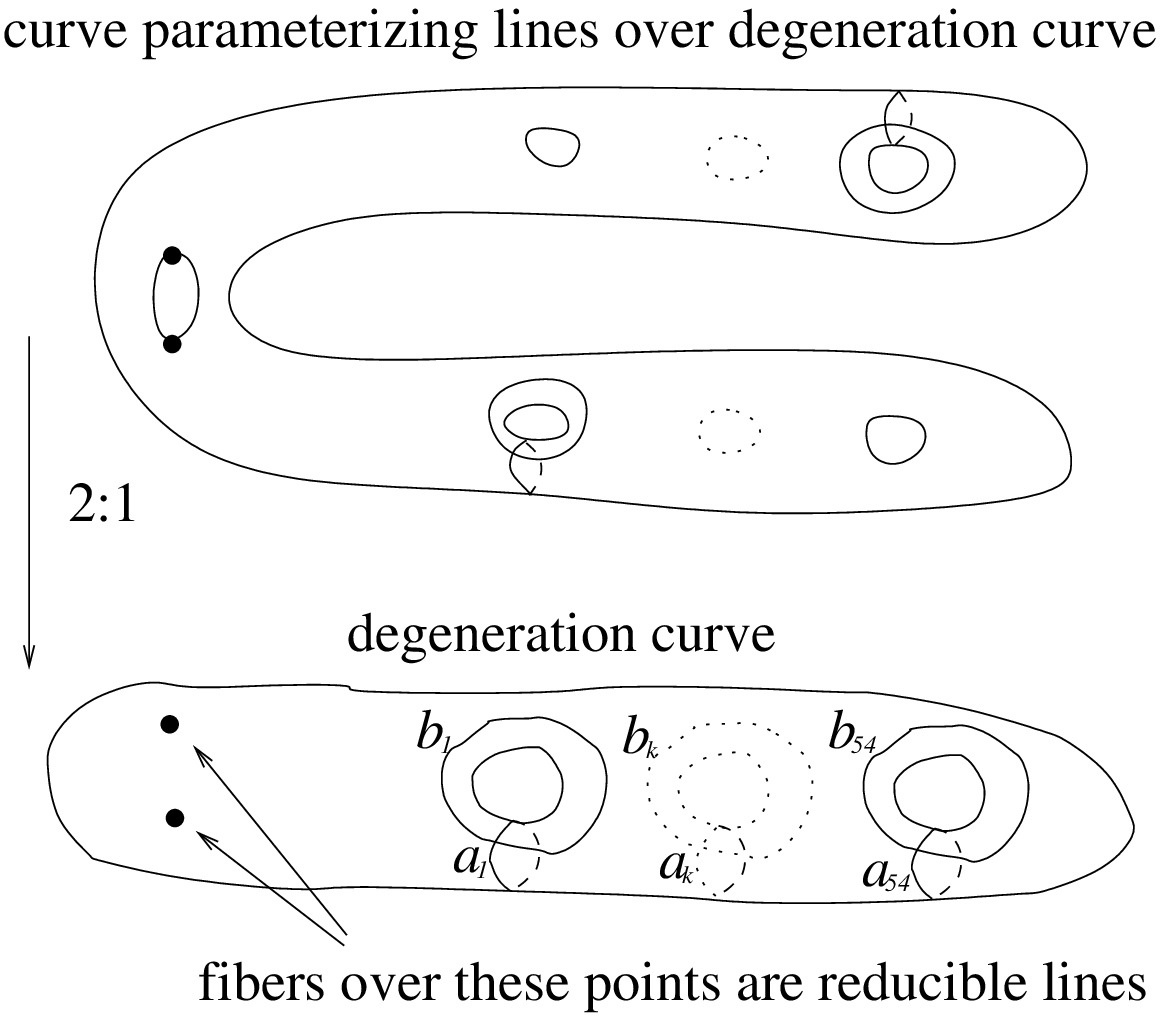}$

$FS(X)/ \langle a_1,\ldots,a_{54}, b_1,\ldots,b_{54} \rangle$ $= \mathcal A\neq 0$, $K^0(\mathcal A)=0$.

\end{center}

\medskip

\end{minipage}
\\ \hline
    \end{tabular}
    \caption{Sarkisov example.}
    \label{tab:CGBT}
  \end{center}
\end{table}

So conjecturally we have
$$
D^b(X)=\langle \mathcal A, E_1, \ldots, E_{112} \rangle,\ \ \
K^0(\mathcal A)=0,\ \ \
\mathcal A\neq 0.
$$

The degeneration construction above suggest an ample opportunity of constructing phantom categories.

\begin{conj}
The SOD of  derived category of  degeneration of  a generic quadric bundles over  a surface contains a phantom category.

\end{conj}

As a result we conjecture nonrationality of such quadric bundles. An interesting question is where this phantom categories come from.
The  analysis of Sarkisov's  example  suggests the following. We start with
a conic bundle  over $\PP^2$  with a curve of degeneration $C$. Such a conic bundle has nontrivial gap in the spectra of its  derived category.
Via  degeneration we reduce this gap in the same way as via degeneration we get rid of the intermediate Jacobian. The degeneration of the two sheeted covering of $\PP^2$ produces a surface with a phantom category in its SOD.
This observation provides us with many possibilities to construct geometric examples of phantom categories. After all many of the classical examples of surfaces of general type are obtained from rational surfaces by taking double coverings, quotioning by group actions, degenerations, and smoothings.

So the  existence of nontrivial categories in SOD 
with trivial Hochschild  homology
can be conjecturally seen in the case of classical surfaces of general type, Campedelli, Godoux (see \cite{BBS}), Burniat (see \cite{AO}), Dolgachev  surfaces, product of curves (\cite{GS12}), and also in the categories of quotients of product of curves and fake $\PP^2$. These surfaces are not rational since they have non-trivial fundamental groups but also since they have conjecturally a quasi-phantom subcategory in their SOD. We call a category a \emph{quasi-phantom} if it is a nontrivial category
with a trivial Hochschild  homology. On the Landau--Ginzburg side these quasi-phantoms are described by the moving scheme.
he deformation of Landau--Ginzburg models  is determined by the moving
scheme so the  (quasi-)phantoms factor in the geometry of the
Landau--Ginzburg  models.

On the mirror side this translates to the fact the Exts between the
(quasi-) phantom category and the rest of the SOD determines the moduli
space.

Finding phantoms --- nontrivial categories with trivial $K^0$ groups ---
is
a quantum leap more difficult than finding   quasi-phantoms.  We
conjecture
that derived categories of  Barlow surface (see \cite{DKK2})  and rational
blow-downs  contain phantoms in their SOD.

%
%

Applying described above quadric bundles construction one conjecturally can produce many examples of phantom categories. There are two main parallels we build our
quadric bundles construction on:

\begin{enumerate}
\item Degenerations of Hodge structures applied to intermediate Jacobians.
This construction goes back to Clemens and Griffiths and later to Alexeev.
They degenerate intermediate Jacobians to Prym varieties or completely to algebraic tori. The important information to remember are the data of degeneration.
For us the algebraic tori is analogous to the phantom and degeneration data to the gap in the spectrum. The data to analyze is how the phantom  fits in SOD.
This determines the gap and the geometry of the Landau--Ginzburg moduli space.
It is directly connected with the geometry of the moving scheme and the monodromy of LG models.

\item The Candelas idea to study rigid Calabi--Yau by including them in Fano e.g. 4-, 7-, 10-dimensional cubics. This gives him the freedom  to deform.
Similarly by including the phantom in the quadric bundles we get the opportunity to deform and degenerate. The reach SOD of the quadric bundle allows us to study the phantom. This of course is a manifestation of the geometry of the moving scheme and the monodromy at infinity of the moduli space of LG models.
\end{enumerate}

 We turn to  the A side and pose the following question.

\begin{ques} Do A side phantoms provide examples of nonsymplectomorphic symplectic manifolds with the same Gromov--Witten invariants?
\end{ques}

\begin{oss} As an initial application of the above conic bundle to classical Horikawa surfaces (see \cite{HOR})  seem to suggest that after deformation we get a phantom in Fukaya category for one of them and not for another one. It would be interesting to see if Hodge type of the argument would lead to the fact that Fukaya categories in these two types of Horikawa surfaces have different gaps of spectra of their Fukaya categories and as a result are not symplectomorphic. It will be analogous to degenerating Hodge structures to non-isomorphic ones for the benefit of geometric consequences.
\end{oss}

In what follows we move to finding quantitative statements for  gaps and phantoms. We have already emphasized the importance of quadric bundles.
In what follows we concentrate on moduli space of a stability conditions
of local CY obtained as quadric bundles.

We move to a second type of invariants we have mentioned. We take the point of view from \cite{HKK} that for special type of Fukaya and Fukaya wrapped categories locally stability conditions are described by differentials with coefficients irrational   or exponential functions. The main idea in \cite{HKK} is that for such categories we can tilt the $t$-structure in a way that the heart  of it becomes an Artinian category. Such a simple $t$-structure allow description of stability condition in terms of geometry of Lefschetz theory and as a result in terms of the moduli space of Landau--Ginzburg models.

As it is suggested in \cite{BFK} there is a connection between monodromy of Landau--Ginzburg models and the gaps of a spectrum. We record our observations in  Tables~\ref{tab:LGS} and~\ref{tab:LOC}.

\begin{enumerate}
  \item  In the case of $A_n$ category the stability conditions are just exponential differentials as we have demonstrated in Example~\ref{example:An}.
  In this case the simple objects for the $t$-structures are given by the intervals connecting singular points of the function given by
  the central charge. 
  \item  Similarly for  one-dimensional Fukaya 
  wrapped categories  the simple objects for the $t$-structures are given by
  the intervals connecting zero sets of the differentials. 
This procedure allows us to take categories with quivers.
  \item  For more complicated Fukaya--Seidel categories obtained as a superposition of one-dimensional  Fukaya 
      wrapped --- see \cite{GKP}
  --- we describe the stability conditions by intertwining the stability conditions for   one-dimensional  Fukaya wrapped 
  ---
  look at the last line of   Table~\ref{tab:LOC}. At the end we obtain a number $d/k$ where
$d$ is the degree of some of the polynomials $p(z)$ involved in the formula and $k$ is the root we take out of it. 
Such a number can be associated with 
Fukaya--Seidel category associated with a local Calabi--Yau manifold obtained as a quadric bundle.
\end{enumerate}

\begin{table}[h]
  \begin{center}
    \begin{tabular}[t]{|c|c|c|}
\hline
\begin{minipage}[c]{0.6\nanowidth}
\centering
\medskip

Category.

\medskip

\end{minipage}
&
\begin{minipage}[c]{0.8\nanowidth}
\medskip

\begin{center}
Imaginary part of a central charge.
\end{center}

\medskip

\end{minipage}
&
\begin{minipage}[c]{\nanowidth}
\medskip

\begin{center}

Landau--Ginzburg models and hearts of a $t$-structures.

\end{center}

\medskip

\end{minipage}
\\
\hline \hline
\begin{minipage}[c]{0.6\nanowidth}
\medskip

\begin{center}

$A_n$

\end{center}

\medskip

\end{minipage}
&
\begin{minipage}[c]{0.8\nanowidth}
\medskip

\begin{center}

$Im \int e^{p(z)}dz$

\end{center}

\medskip

\end{minipage}
&
\begin{minipage}[c]{\nanowidth}
\medskip

\begin{center}

$\includegraphics[width=3.5cm]{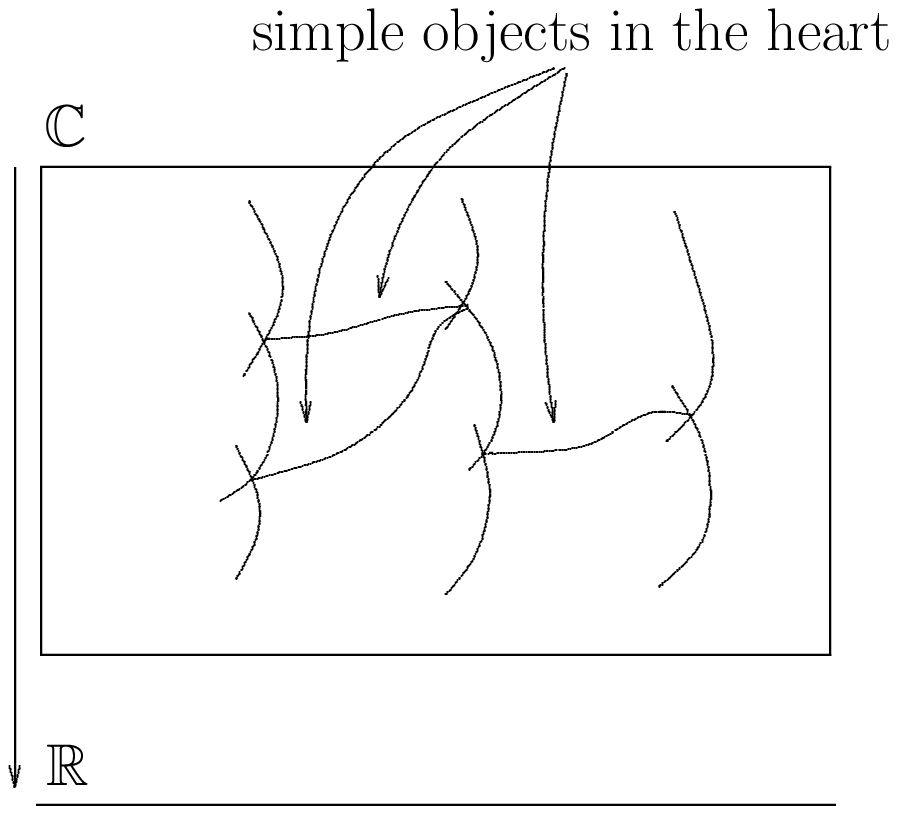}$

\end{center}

\medskip

\end{minipage}
\\ \hline
\begin{minipage}[c]{0.6\nanowidth}
\medskip

\begin{center}

1-dimensional Fukaya wrapped category or
Fukaya category of $d$-dimensional Calabi--Yau category.

\end{center}

\medskip

\end{minipage}
&
\begin{minipage}[c]{0.8\nanowidth}
\medskip

\begin{center}

$Im \int q(z)dz$, where $q(z)$ is a quadratic differential with all zeros of multiplicities $d-2$.

\end{center}

\medskip

\end{minipage}
&
\begin{minipage}[c]{\nanowidth}
\medskip

\begin{center}

$\includegraphics[width=3.5cm]{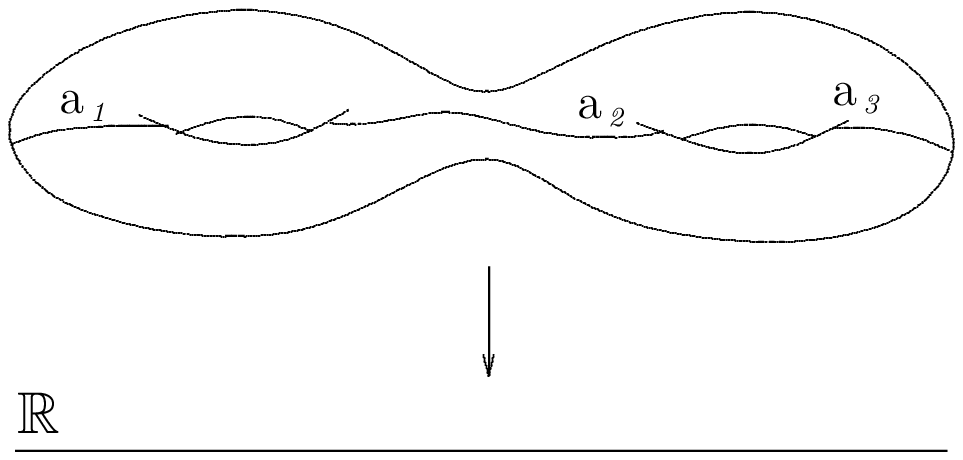}$

$a_i$'s are simple objects.

\end{center}

\medskip

\end{minipage}
\\ \hline
\begin{minipage}[c]{0.6\nanowidth}
\medskip

\begin{center}

$D^b(\PP^n)$

\end{center}

\medskip

\end{minipage}
&
\begin{minipage}[c]{0.8\nanowidth}
\medskip

\begin{center}

$Im z\colon K^0(D^b(\PP^n))\to \CC$

\end{center}

\medskip

\end{minipage}
&
\begin{minipage}[c]{\nanowidth}
\medskip

\begin{center}

The $\PP^n$ quiver.

\end{center}

\medskip

\end{minipage}
\\ \hline
\begin{minipage}[c]{0.6\nanowidth}
\medskip

\begin{center}

$D^b(X)$, $X$ is a Fano variety.

\end{center}

\medskip

\end{minipage}
&
\begin{minipage}[c]{0.8\nanowidth}
\medskip

\begin{center}

Moduli space of Landau--Ginzburg models.

\end{center}

\medskip

\end{minipage}
&
\begin{minipage}[c]{\nanowidth}
\medskip

\begin{center}

Intertwining Landau--Ginzburg models and hearts via monodromy  --- creating gaps in spectra.
\end{center}

\medskip

\end{minipage}
\\ \hline
    \end{tabular}
    \caption{Landau--Ginzburg models and stability conditions.}
    \label{tab:LGS}
  \end{center}
\end{table}

\begin{table}[h]
  \begin{center}
    \begin{tabular}[t]{|c|c|}
\hline
\begin{minipage}[c]{1.5\nanowidth}
\centering
\medskip

Fukaya categories. 

\medskip

\end{minipage}
&
\begin{minipage}[c]{1.3\nanowidth}
\medskip

\begin{center}

Stability conditions.

\end{center}

\medskip

\end{minipage}
\\
\hline \hline
\begin{minipage}[c]{1.5\nanowidth}
\medskip

\begin{center}

$\Fuk(F)$, where
$F$ is a local Calabi--Yau variety $\{y^2+r(x)=\alpha_1^2+\ldots+\alpha_{d-2}^2\}$,
where $p$ is a polynomial of degree $n+1$.

\end{center}

\medskip

\end{minipage}
&
\begin{minipage}[c]{1.3\nanowidth}
\medskip

\begin{center}

$q(z)$, where $q$ is a quadratic differential with roots of multiplicities $d$.

\end{center}

\medskip

\end{minipage}
\\ \hline
\begin{minipage}[c]{1.5\nanowidth}
\medskip

\begin{center}


$\Fuk_{wrapped}(C)(d)$, where $C$ is a Riemann surface with punctures.

\end{center}

\medskip

\end{minipage}
&
\begin{minipage}[c]{1.3\nanowidth}
\medskip

\begin{center}

$q(z)e^{p(z)}dz$, where $q$ is a quadratic differential with roots of multiplicities $d$ and $p$ is a polynomial
of degree $2g(C)+1$.

\end{center}

\medskip

\end{minipage}
\\ \hline
\begin{minipage}[c]{1.5\nanowidth}
\medskip

\begin{center}


$\left(\Fuk (F_1)\times \ldots\times\Fuk(F_m)\right.$ $\times \Fuk_{wrapped}(C_1)\times \ldots$
$\left.\times \Fuk_{wrapped}(C_n)\right)/\ZZ_k$

\end{center}

\medskip

\end{minipage}
&
\begin{minipage}[c]{1.3\nanowidth}
\medskip

\begin{center}

$\sqrt[k]{q_1(z)\cdot\ldots\cdot q_{m+n}(z)}$ $\cdot e^{p_1(z)+\ldots+p_n(z)}dz$.

\end{center}

\medskip

\end{minipage}
\\ \hline
    \end{tabular}
    \caption{Conjectural duality.}
    \label{tab:LOC}
  \end{center}
\end{table}

The following conjecture suggests  local invariants.

\begin{conj} The geometry of the moving set determines the number
$d/k$ and the gap of the spectra of the corresponding category.
\end{conj}

This conjecture suggests the following question.

\begin{ques} Is the number $d/k$ a birational invariant?
\end{ques}

In Table~\ref{tab:LOC} we give examples of categories 
and their stability conditions.
These categories serve as building blocks for more involved categories of Fano manifolds.

In case the question above has  a  positive answer we get a way of comparing the moving schemes of the divisors at infinity. We will also get a way of deciding if the corresponding moduli spaces of Landau--Ginzburg models can be deformed to each other which according to~\cite{DKK} is a way of deciding if we can build a Sarkisov link
between them.

%
%
%
%

It is clear that the numbers   $d/k$ fit well in the landscape of quadric bundles. We expect particularly interesting behavior from the numbers $d/k$
coming from the phantom categories the existence of which was conjectured  earlier.

\begin{ques} Can we read the  existence phantoms and gaps in terms of the
numbers $d/k$?
\end{ques}

It is also clear  that the geometry of the moving scheme in
general  has a deep  connection with  the geometry of the Fano
manifold. In fact following  \cite{STEP} we associate a complex of
singularities with this moving scheme. So it is natural expect
that we can read many geometrical properties of Fano manifolds
from this complex of singularities. For example, it has been
conjectured by S.-T.\,Yau, G.\,Tian, and S.\,Donaldson that some
kind of \emph{stability} of Fano manifolds is a necessary and
sufficient condition for the existence of K\"ahler--Einstein
metrics on them. This conjecture has been verified in
two-dimensional case (see \cite{Ti90}) and in the toric case (see
\cite{WaZhu04}). Moreover, one direction of this conjecture is now
almost proved by Donaldson, who showed that the existence of the
K\"ahler--Einstein metric implies the so-called $K$-semistability
(see \cite{Donaldson2010}, \cite{Donaldson2011a}). Recall that
G.\,Tian (see~\cite{Ti97}) defined the notion of $K$-stability,
arising from certain degenerations of the manifold or, as he
called them, test configurations. Proving Yau--Tian--Donaldson
conjecture is currently a major research programme in Differential
Geometry (see \cite{Donaldson2009}). We finish with  the following
question.

\begin{ques} Can we read the  existence of K\"ahler--Einstein metric on the Fano manifold  from this complex of singularities?
\end{ques}


\begin{thebibliography}{57}

\bibitem[AAK08]{AAK} {\sc Abouzaid M. {\rm and} Auroux D. {\rm and} Katzarkov L.}, \newblock
  \emph{Homological mirror symmetry for blowups}, preprint, 2008.

\bibitem[Al02]{Al02} {\sc Alexeev V.}, \newblock
\emph{Complete moduli in the presence of semiabelian group action},
 Ann. of Math. (2) 155 (2002), No. 3, 611--708, arXiv:math/9905103.

\bibitem[AO12]{AO}
{\sc Alexeev V. {\rm and} Orlov D.},
{\it Derived categories of Burniat surfaces and exceptional collections},
arXiv:1208.4348.


%

%

\bibitem[BFK10]{BFK}
{\sc Ballard M. {\rm and} Favero D. {\rm and} Katzarkov L.}, \emph{The Orlov spectrum:
gaps and bounds}, submitted Inventiones
Mathematicae, {arXiv:1012.0864}.

\bibitem[BFK11]{BFK2}
{\sc Ballard M. {\rm and} Favero D. {\rm and} Katzarkov L.}, \emph{A category of
kernels for graded matrix factorizations and its implications
towards Hodge theory}, submitted to Publications
math\'ematiques de l'IH\'ES, {arXiv:1105.3177}. 

%

%
%
%
%
%


%
%
%

\bibitem[Ba94]{Batyrev1994}
{\sc Batyrev V.},  \emph{Dual polyhedra and mirror symmetry for Calabi--Yau hypersurfaces in toric varieties},
Journal of Algebraic Geometry \textbf{3} (1994), 493--535, arXiv:alg-geom/9310003.
%
%
%
%
%

\bibitem[Ba04]{Batyrev2004}
{\sc Batyrev V.},  \emph{Toric degenerations of Fano varieties and constructing mirror manifolds},
Prooceedings of the Fano conference, University of Torino (2004), 109--122, arXiv:alg-geom/9712034.

%

\bibitem[Ba12]{Batyrev2012}
{\sc Batyrev V.},  \emph{Conifold degenerations of Fano 3-folds as hypersurfaces in toric varieties},
arXiv:1203.6058.

\bibitem[BB95]{BB}
{\sc Batyrev V. {\rm and} Borisov L.}, \emph{Dual Cones and Mirror Symmetry
for Generalized Calabi--Yau Manifolds}, in Mirror Symmetry II,
(eds. S.-T. Yau), 65--80 (1995),  arXiv:alg-geom/9402002.

%
%
%
%
%
%
%
%
%
%
%
%
%
%
%
%
%
%

\bibitem[BCHM06]{BCHM06}
{\sc Birkar C. {\rm and} Cascini P. {\rm and} Hacon C. {\rm and} McKernan J.}, \emph{Existence of minimal models for varieties of log general type},
J. Am. Math. Soc. 23, No. 2, 405--468 (2010), arXiv:math/0610203.

\bibitem[BL11]{BlancLamy}
{\sc Blanc J. {\rm and} Lamy S.}, \emph{Weak Fano threefolds obtained by blowing-up a space curve and construction of Sarkisov links},
arXiv:1106.3716.

%
%
%
%
%
%
%
%
%
%

\bibitem[BBS12]{BBS}
{\sc Bohning C. {\rm and} von Bothmer H.-C. {\rm and} Sosna P.},
{\it On the derived category of the classical Godeaux surface},
arXiv:1206.1830.

\bibitem[Bri02]{B}
{\sc Bridgeland T.}, \emph{Stability conditions on triangulated categories}, Ann. of Math. (2)  166  (2007),  no. 2, 317--345, arXiv:math/0212237.

\bibitem[Bro07]{Br}
{\sc Brown G.}, \emph{A database of polarized $K3$ surfaces},
Experimental Mathematics, \textbf{16} (2007), 7--20.


%
%
%

\bibitem[Ch96]{Ch96}
{\sc Cheltsov I.}, \emph{Three-dimensional algebraic manifolds having a divisor with a numerically trivial canonical class},
Russian Mathematical Surveys, \textbf{51} (1996), 140--141.


%
%


\bibitem[Ch03]{Cheltsov03}
{\sc Cheltsov I.}, \emph{Anticanonical models of three-dimensional Fano varieties of degree four},
Sbornik: Mathematics, \textbf{194} (2003), 147--172.


%

%
%
%
%
%
%
%

\bibitem[CPS04]{ChPrSh2005}
{\sc Cheltsov I. {\rm and}  Przyjalkowski V. {\rm and} Shramov C.}, \emph{Hyperelliptic and trigonal Fano threefolds},
Izvestiya: Mathematics, \textbf{69} (2005), 365--421, arXiv:math/0406143.

%
%
%
%
%
%
%
%
%
%
%
%
%
%
%
%
%
%
%
%
%
%
%
%
%
%
%
%

\bibitem[CI11]{CI11}
{\sc Christophersen J. {\rm and} Ilten N.}, \emph{Stanley--Reisner degenerations of Mukai varieties},
arXiv:1102.4521.

\bibitem[CI12]{CI12}
{\sc Christophersen J. {\rm and} Ilten N.}, \emph{Toric degenerations of low degree Fano threefolds},
arXiv:1202.0510.


\bibitem[CV91]{CV}
{\sc Cecotti S. {\rm and} Vafa C.}, \emph{Topological-anti--topological
fusion}, Nuclear Phys. B  367  (1991),  no. 2, 359--461.


%
%
%
%
%
%


\bibitem[CCGGK]{fanosearch}
{\sc Coates T. {\rm and} Corti A. {\rm and} Galkin S. {\rm and} Golyshev V. {\rm and} Kasprzyk A.}, \emph{Fano varieties and extremal Laurent polynomials}, a collaborative research blog, \url{http://coates.ma.ic.ac.uk/fanosearch}.



\bibitem[Co95]{Co95}
{\sc Corti A.}, \emph{Factorizing birational maps of threefolds after Sarkisov},
Journal of Algebraic Geometry, \textbf{4} (1995), 223--254.

\bibitem[CG06]{ALESS}
{\sc Corti A. {\rm and} Golyshev V.},
\emph{Hypergeometric equations and weighted projective spaces},
Sci. China, Math. 54, No. 8, 1577--1590 (2011), arXiv:math/0607016.

\bibitem[CPR00]{CPR}
{\sc Corti A. {\rm and} Pukhlikov A. {\rm and} Reid M.}, \emph{Fano 3-fold
hypersurfaces}, L.M.S. Lecture Note Series \textbf{281} (2000),
175--258.
%
%
%
%
%

\bibitem[CM10]{CutroneMarshburn}
{\sc Cutrone J. {\rm and} Marshburn N.}, \emph{Towards the classification of weak Fano threefolds with $\rho=2$},
arXiv:math.AG/1009.5036.

%
%
%
%

\bibitem[De80]{Demazure}
{\sc Demazure M.}, \emph{Surfaces de Del Pezzo. I-V.},
Seminaire sur les singularites des surfaces, Lecture Notes in Mathematics, \textbf{777}, 21--69 (1980).

\bibitem[DKK12a]{DKK1}
{\sc Diemer C. {\rm and} Katzarkov L. {\rm and} Kerr G.}, \emph{Symplectic relations
arising from toric degenerations}, arXiv:1204.2233.

\bibitem[DKK12b]{DKK2}
{\sc Diemer C. {\rm and} {\rm and} Katzarkov L. {\rm and} Kerr G.}, \emph{Compactifications of spaces of Landau--Ginzburg models},
arXiv:1207.0042.

\bibitem[DKK]{DKK}
{\sc Diemer C. {\rm and} {\rm and} Katzarkov L. {\rm and} Kerr G.},
 \newblock \emph{Stability conditions for FS categories}, in preparation.

%

%
%
%
%
%
%

\bibitem[Do09]{Donaldson2009}
{\sc Donaldson S.}, \emph{Discussion of the K\"ahler--Einstein problem},
preprint, \url{http://www2.imperial.ac.uk/~skdona/KENOTES.PDF}.

\bibitem[Do10]{Donaldson2010}
{\sc Donaldson S.}, \emph{Stability, birational transformations and the~Kahler--Einstein problem},
arXiv:1007.4220.

\bibitem[Do11]{Donaldson2011a}
{\sc Donaldson S.}, \emph{$b$-Stability and blow-ups},
arXiv:1107.1699.

\bibitem[DT96]{DT}
{\sc Donaldson S. {\rm and} Thomas R.},
\emph{Gauge theory in higher dimensions},  The geometric universe (Oxford, 1996),  31--47, Oxford Univ. Press, Oxford, 1998.

\bibitem[DKLP]{DKLP}
{\sc Doran C. {\rm and} Katzarkov L. {\rm and} Lewis J. {\rm and} Przyjalkowski V.},
\emph{Modularity of Fano threefolds}, in preparation.

\bibitem[Do00]{D}
{\sc Douglas M.}, \emph{D-branes, categories and $N = 1$ supersymmetry. Strings, branes, and M-theory}, J. Math. Phys. 42 (2001), 2818--2843, arXiv:hep-th/0011017.

%
%
%
%
%
%
%
%
%

\bibitem[Fa34]{Fa34} {\sc Fano G.}, \emph{Sulle varieta algebriche a tre
dimensioni aventi tutti i generi nulu}, Proc. Internat. Congress
Mathematicians (Bologna), 4, Zanichelli, 115--119 (1934).

\bibitem[Fa42]{Fa42}
{\sc Fano G.}, \emph{Su alcune varieta algebriche a tre dimensioni
razionali, e aventi curve-sezioni canoniche}, Commentarii
Mathematici Helvetici 14: 202--211 (1942).

\bibitem[Ga]{Ga}
{\sc Galkin S.}, \emph{Small toric degenerations of Fano threefolds},
preprint, \url{http://sergey.ipmu.jp/papers/std.pdf}.

\bibitem[GS12]{GS12}
{\sc Galkin S. {\rm and} Shinder E.}, \emph{Exceptional collections of line bundles on the Beauville surface}, arXiv:1210.3339.

\bibitem[GMN08]{GMN}
{\sc Gaiotto D. {\rm and} Moore G. {\rm and} Neitzke A.},
\emph{Four-dimensional wall-crossing via three-dimensional field theory},
Comm. Math. Phys.  299  (2010),  no. 1, 163--224, arXiv:0807.4723.
%
%
%

\bibitem[GKZ94]{GKZ}
{\sc Gelfand I. {\rm and} Kapranov M. {\rm and} Zelevinski A.},
\emph{Discriminants, resultants and multidimensional determinants},
Mathematics: Theory and Applications. Birkhauser Boston, Inc., Boston, MA, 1994.

\bibitem[Go02]{Golyshev2002}
{\sc Golyshev V.}, \emph{The geometricity problem and modularity of some Riemann--Roch variations},
Russian Academy of Sciences (Doklady, Mathematics), \textbf{386}, 583--588 (2002).

\bibitem[Go05]{Golyshev2007}
{\sc Golyshev V.}, \emph{Classification problems and mirror duality},
LMS Lecture Note Series, \textbf{338}, 88--121 (2007), arXiv:math/0510287.

\bibitem[HM]{HaconMcKernan}
{\sc Hacon C. {\rm and} McKernan J.}, \emph{The Sarkisov program},
arXiv:0905.0946.

%
%
%
%
%
%
%
%

\bibitem[HKK]{HKK}
{\sc Haiden F. {\rm and} Katzarkov L. {\rm and} Kontsevich M.},
 \newblock \emph{Stability conditions for FS categories}, in prepation.

%
\bibitem[Ha99]{HASS}
{\sc Hassett B.}, \emph{Some rational cubic fourfolds}, J. Algebraic
Geom., \textbf{8} (1999), 103--114.

\bibitem[HV00]{hori-vafa}
{\sc Hori K. {\rm and} Vafa C.}, \emph{Mirror symmetry}, hep-th/0002222.

\bibitem[Ho76]{HOR}
{\sc Horikawa E.},
{\it Algebraic surfaces of general type with small $c_1^2$. I}, Ann. Math. (2) 104, 357--387 (1976).

\bibitem[IK10]{IKUZ}
{\sc Ingalls C. {\rm and} Kuznetsov A.},
{\it On nodal Enriques surfaces and quartic double solids}, arXiv:1012.3530.

\bibitem[IKP11]{IKP}
{\sc Iliev A. {\rm and} Katzarkov L. {\rm and} Przyjalkowski V,},
{\it Double solids, categories and non-rationality},
PEMS (2014) 57, 145--173, arXiv:1102.2130.

\bibitem[IKS]{IKS} {\sc Iliev A. {\rm and} Katzarkov L. {\rm and} Scheidegger E.},
\emph{Automorphic forms and cubics}, in preparation.


\bibitem[ILP11]{ILP11}
{\sc Ilten N. {\rm and} Lewis J. {\rm and} Przyjalkowski V.},
{\emph{Toric degenerations of Fano threefolds giving weak Landau--Ginzburg models}}, Journal of Algebra 374 (2013), 104--121., arXiv:1102.4664 (2011).

\bibitem[IV09]{IV09} {\sc Ilten N. {\rm and} Vollmert R.}, \emph{Deformations of Rational T-Varieties},
To appear in Journal of
Algebraic Geometry, {arXiv:0903.1393}.


\bibitem[Is77]{Is77}
{\sc Iskovskikh V.}, \emph{Fano 3-folds I},
Mathematics of the USSR, Izvestija, \textbf{11} (1977), 485--527.



\bibitem[Is78]{Is78}
{\sc Iskovskikh V.}, \emph{Fano 3-folds II},
Mathematics of the USSR, Izvestija, \textbf{12} (1978), 469--506.

\bibitem[IM71]{IsMa71}
{\sc Iskovskikh V. {\rm and} Manin Yu.}, \emph{Three-dimensional quartics and counterexamples to the L\"uroth problem},
Matematical Sbornik, \textbf{86} (1971), 140--166.

\bibitem[IP99]{IsPr99}
{\sc Iskovskikh V. {\rm and} Prokhorov Yu.}, \emph{Fano varieties},
Encyclopaedia of Mathematical Sciences, \textbf{47} (1999) Springer, Berlin.


\bibitem[JPR04]{JaPeRa07}
{\sc Jahnke P. {\rm and} Peternell T. {\rm and} Radloff I.}, \emph{Threefolds with big and nef anticanonical bundles I},
Mathematische Annalen, \textbf{333} (2005), 569--631, arXiv:math/0407484.

\bibitem[JPR07]{JaPeRa09}
{\sc Jahnke P. {\rm and} Peternell T. {\rm and} Radloff I.}, \emph{Threefolds with big and nef anticanonical bundles II},
 Cent. Eur. J. Math.  9  (2011),  no. 3, 449--488,  arXiv:0710.2763.

\bibitem[JR04]{JaRa06}
{\sc Jahnke P. {\rm and} Radloff I.}, \emph{Gorenstein Fano threefolds with base points in the anticanonical system},
Compositio Mathematica, \textbf{142} (2006), 422--432, arXiv:math/0404156.

\bibitem[JR06]{JR11}
{\sc Jahnke P. {\rm and} Radloff I.}, \emph{Terminal Fano threefolds and their smoothings}.
Mathematische Zeitschrift \textbf{269} (2011), 1129--1136, arXiv:math/0601769.


%
%
%
%
%
%
%
%
%

\bibitem[JS08]{JS}
{\sc Joyce D. {\rm and} Song Y.}, \emph{A theory of generalized Donaldson--Thomas invariants},
Mem. Amer. Math. Soc.  217  (2012),  no. 1020, arXiv:0810.5645.

\bibitem[Ka07]{Kaloghiros}
{\sc Kaloghiros A.-S.}, \emph{The topology of terminal quartic $3$-folds},
arXiv:math.AG/0707.1852.

\bibitem[KKOY09]{KKOY}
{\sc Kapustin A. {\rm and} Katzarkov L. {\rm and} Orlov D. {\rm and} Yotov M.}, \emph{Homological Mirror Symmetry for manifolds of general type},
Cent. Eur. J. Math. 7, No. 4, 571--605 (2009), arXiv:1004.0129.

%
%
%
%


\bibitem[Ka08]{Karzhemanov2008}
{\sc Karzhemanov I.}, \emph{On Fano threefolds with canonical Gorenstein singularities},
Sbornik: Mathematics, \textbf{200} (2009), 111--146, arXiv:0805.3927.



\bibitem[Ka09]{Karzhemanov2009}
{\sc Karzhemanov I.}, \emph{Fano threefolds with canonical Gorenstein singularities and big degree},
arXiv:0908.1671.


%

\bibitem[KKP]{KKP}
{\sc Katzarkov L. {\rm and} Kontsevich M. {\rm and} Pantev T.}, \newblock \emph{Hodge
theoretic aspects of mirror symmetry 2.}

\bibitem[KKPS]{KKPS}
{\sc Katzarkov L. {\rm and} Kontsevich M. {\rm and} Pantev T. {\rm and} Soibelman Y.},
\newblock \emph{Shability Hodge structures.}, in preparation.

\bibitem[KP09]{KP} {\sc Katzarkov L. {\rm and} Przyjalkowski V.},
{\it Generalized Homological Mirror Symmetry and cubics}, Proc. Steklov Inst. Math., vol. 264, 2009, 87--95,
\url{http://www.mi.ras.ru/~victorprz/rus/katprz.pdf}.

\bibitem[KP12]{GKP} {\sc Katzarkov L. {\rm and} Przyjalkowski V.},
\newblock \emph{Landau--Ginzburg models --- old and new},
\newblock Akbulut, Selman (ed.) et al., Proceedings of the 18th Gokova geometry--topology conference.
Somerville, MA: International Press; Gokova: Gokova Geometry--Topology Conferences, 97--124 (2012).




%
%
%
%
%
%
%
%
%
%
%
%
%
%
%
%
%
%
%
%
%
%
%
%
%

\bibitem[KS09]{KS2}
{\sc Kontsevich M. {\rm and} Soibelman Y.}, \emph{Motivic Donaldson--Thomas invariants:
summary of results},  Mirror symmetry and tropical geometry,  55--89, Contemp. Math., 527, Amer. Math. Soc., Providence, RI, 2010, arXiv:0910.4315.

\bibitem[KS10]{KS1}
{\sc Kontsevich M. {\rm and} Soibelman Y.}, \emph{Cohomological Hall algebra, exponential Hodge structures and motivic Donaldson--Thomas invariants}. Commun. Number Theory Phys.  5  (2011),  no. 2, 231--352, arXiv:1006.2706.


\bibitem[KS95]{KreuzerSkarke1997}
{\sc Kreuzer M. {\rm and} Skarke H.}, \emph{On the classification of reflexive
polyhedra}, Communications in Mathematical Physics,  \textbf{185}
(1997), 495--508, arXiv:hep-th/9512204.

\bibitem[KS98]{KreuzerSkarke1998}
{\sc Kreuzer M. {\rm and} Skarke H.}, \emph{Classification of reflexive
polyhedra in three dimensions}, Advances in Theoretical and
Mathematical Physics,  \textbf{2} (1998), 853--871, arXiv:hep-th/9805190.

\bibitem[Ku95]{KUECH}
{\sc Kuchle O.}, \emph{On Fano $4$-folds of index $1$ and homogeneous
vector bundles over Grassmannians}, Mathematische Zeitschrift
\textbf{218} (1995), 563--575.

\bibitem[Ku08]{KUZ}
{\sc Kuznetsov A.}, \emph{Derived categories of cubic fourfolds},
Cohomological and geometric approaches to rationality problems,
Progress in Mathematics, \textbf{282}, Birkhauser Boston (2010),
219--243, arXiv:0808.3351.


%
%
%
%
%
%
%
%
%
%
%
%
%
%
%
%
%
%
%

\bibitem[La98]{Lafforgue}
{\sc Lafforgue L.},
\emph{Une compactification des champs classifiant les chtoucas de Drinfeld},
J. Amer. Math. Soc., 11 (1998), no. 4, 1001--1036.

\bibitem[MNOP03]{MNOP1}
{\sc Maulik D. {\rm and} Nekrasov N. {\rm and} Okounkov A. {\rm and} Pandharipande R.},
\emph{Gromov--Witten
theory and Donaldson--Thomas theory. I}, Compos. Math.  142  (2006),  no. 5, 1263--1285, arXiv:math/0312059.

\bibitem[MNOP04]{MNOP2}
{\sc Maulik D. {\rm and} Nekrasov N. {\rm and} Okounkov A. {\rm and} Pandharipande R.},
\emph{Gromov--Witten theory and Donaldson--Thomas theory. II}, Compos. Math.  142  (2006),  no. 5, 1286--1304,
arXiv:math/0406092.
%
%
%

%
%

%
%
%
%
%
%
\bibitem[Mu02]{Mukai}
{\sc Mukai S.}, \emph{New developments in the theory of Fano
threefolds: vector bundle method and moduli problems}, Sugaku
Expositions, \textbf{15} (2002), 125--150.
%
%
%
%
%

\bibitem[Na97]{Na97}
{\sc Namikawa Y.}, \emph{Smoothing Fano 3-folds},
Journal of Algebraic Geomemtry, \textbf{6} (1997), 307--324.

%
%
%
%
%
%
%
%
%
%
%
%
%
\bibitem[Or08]{O}
{\sc Orlov D.}, \emph{Remarks on generators and dimensions of triangulated categories}, Mosc. Math. J.  9  (2009),  no. 1, 153--159, arXiv:0804.1163.
%
%
%
%
%
%


\bibitem[dP87]{dP87}
{\sc del Pezzo P.}, \emph{Sulle superficie dell nno ordine immerse
nello spazio di n dimensioni}, Rend. del circolo matematico di
Palermo 1 (1): 241--271, 1887.

\bibitem[Pr05]{Pr05}
{\sc Prokhorov Yu.}, \emph{The degree of Fano threefolds with canonical Gorenstein singularities},
Sbornik: Mathematics, \textbf{196} (2005), 77--114.

%
%
%
%

\bibitem[Prz07]{Prz07}
{\sc Przyjalkowski V.}, \emph{On Landau--Ginzburg models for Fano varieties},
Communications in Number Theory and Physics, \textbf{1} (2008), 713--728, arXiv:0707.3758.



\bibitem[Prz09]{Prz09}
{\sc Przyjalkowski V.}, \emph{Weak Landau--Ginzburg models for smooth Fano threefolds}, Izv. Math. Vol., 77 No. 4 (2013), 135--160, arXiv:0902.4668.

\bibitem[Re87]{YPG}
{\sc Reid M.}, \emph{Young person's guide to canonical singularities},
Proceedings of the Symposium in Pure Mathematics, \textbf{46} (1987), 345--414.


%
%
%
%
%
%
%
%
%
\bibitem[Ro03]{R}
{\sc Rouquier R.}, \emph{Dimensions of triangulated categories}, J. K-Theory  1  (2008),  no. 2, 193--256, arXiv:math/0310134.

%
%
\bibitem[Sa82]{SAR}
{\sc Sarkisov V.}, \emph{On conic bundle structures},
Izv. Akad. Nauk SSSR Ser. Mat., \textbf{46}:2 (1982), 371--408.


%
%
%
%
\bibitem[SW94]{SW}
{\sc Seiberg N. {\rm and} Witten, E.},
\emph{Electric-magnetic duality, monopole condensation, and confinement in $N=2$ supersymmetric Yang--Mills theory},
Nuclear Phys. B  426  (1994),  no. 1, 19--52.
%
%
%
%
%
%
%
%
%
%
%
%
%
%
%

\bibitem[Si92]{Si92}
{\sc Simpson C.}, \emph{Higgs bundles and local systems},
Publ. Math., Inst. Hautes Etud. Sci. 75, 5--95 (1992).

%
%
%
%

\bibitem[St06]{STEP}
{\sc Stepanov D.},
{\it Combinatorial structure of exceptional sets in resolutions of singularities},  arXiv:math/0611903.

\bibitem[Ta89]{Takeuchi}
{\sc Takeuchi K.}, \emph{Some birational maps of Fano $3$-folds},
Compositio Mathematica, \textbf{71} (1989), 265--283.

%
%
%


\bibitem[Ti90]{Ti90}
{\sc Tian G.}, \emph{On Calabi's conjecture for complex surfaces with positive first Chern class},
Inventiones Mathematicae, \textbf{101} (1990), 101--172.

\bibitem[Ti97]{Ti97}
{\sc Tian G.}, \emph{K\"ahler--Einstein metrics with positive scalar curvature},
Inventiones Mathematicae, \textbf{130} (1997), 1--37.

%
%
%

\bibitem[WZ04]{WaZhu04}
{\sc Wang X. {\rm and} Zhu X.}, \emph{K\"ahler--Ricci solitons on toric manifolds with positive first Chern class},
Advances in Mathematics, \textbf{188} (2004), 87--103.

\end{thebibliography}
\end{document}